%% file: arxiv.tex
\documentclass[]{article}

\RequirePackage[breaklinks,colorlinks,citecolor=blue,urlcolor=blue]{hyperref}
\newcommand{\predh}{\ensuremath{{\widehat{\text{p}}(j)}}}
\newcommand{\predz}{\ensuremath{{{\text{p}}(j)}}}
\usepackage{amssymb}
\usepackage{graphicx}
\usepackage{makeidx}
\usepackage{multicol}
\usepackage{booktabs}
\usepackage{amsfonts}
\usepackage{amsmath}
\usepackage{amsthm}
\usepackage{graphicx}
\usepackage{amssymb}
\usepackage{enumerate}
\usepackage{multicol}
\usepackage{multirow}
\RequirePackage[longnamesfirst,round]{natbib} 

\newtheorem{theorem}{Theorem}
\newtheorem{remark}{Remark}

\newtheorem{lemma}{Lemma}
\newtheorem{assump}{Assumption}
\newtheorem{condition}{Condition}
\newtheorem{proposition}{Proposition}

\newcommand{\vertiii}[1]{{\left\vert\kern-0.25ex\left\vert\kern-0.25ex\left\vert #1
    \right\vert\kern-0.25ex\right\vert\kern-0.25ex\right\vert}}

\title{\bf Inference in high-dimensional graphical models}
 \author{ 
  \small \bf Jana Jankov\'{a} \and  \small \bf Sara van de Geer 
       }
			\date{\small\it Seminar for Statistics\\ETH Z\"urich
}
\newcommand{\conds}{Conditions } 
\newcommand{\condi}{Condition } 

\usepackage{enumerate}

\newtheorem{cond}{Condition}
\newtheorem{condb}{Condition}

\begin{document}
\maketitle

\input{ch1}

\bibliography{gminf}

\end{document}

%% file: ch1.tex
\begin{abstract}

We provide a selected overview of methodology and theory for estimation and inference on the edge weights in high-dimensional \emph{directed} and \emph{undirected} Gaussian graphical models. 
For undirected graphical models, two main explicit constructions are provided: one based on a global method that maximizes the joint likelihood (the graphical Lasso) and one based on a local (nodewise) method that sequentially applies the Lasso to estimate the neighbourhood of each node. 
The estimators lead to confidence intervals for edge weights and recovery of
the edge structure. We evaluate their empirical performance in an extensive simulation study.
The theoretical guarantees for the methods are achieved under a sparsity condition relative to
the sample size and regularity conditions. For directed acyclic graphs, we apply similar ideas
 to construct confidence intervals for edge weights, when the directed acyclic graph is identifiable.
\end{abstract}

\section{Undirected graphical models}
\label{sec:ugm}

\subsection{Introduction}
\par
Undirected graphical models, also known as Markov random fields, have become a popular tool for representing network structure of high-dimensional data  in a large variety  of areas including genetics, brain network analysis, social networks and climate studies. 
Let $G=(\mathcal V,\mathcal E)$ be an undirected graph with a vertex set $\mathcal V = \{1,2,\dots,p\}$ and an edge set $\mathcal E \subset \mathcal V\times \mathcal V$. 
Let $X^0 = (X_1 ,X_2 ,...,X_p )$ be a random vector indexed by graph's vertices. The joint distribution of $X^0$ belongs to the graphical model determined by $G$ if $X_j$ and $X_k$ are conditionally independent  given all other variables whenever $j$ and $k$ are not adjacent in $G.$ 
The graph then encodes conditional independence structure among the entries of $X^0$.

If we assume that $X^0$ is normally-distributed with a covariance matrix $\Sigma_0$, 
one can show that the edge structure of the graph is encoded by the precision matrix $\Theta_0:=\Sigma_0^{-1}$ (assumed to exist). If $\Theta_{ij}^0$ denotes the $(i,j)$-th entry of the matrix $\Theta_0$, it is well known that $\Theta_{ij}^0 =0 \Leftrightarrow (i,j)\not \in \mathcal E.$ The non-zero entries in the precision matrix correspond to edges in the associated graphical model and the absolute values of these entries correspond to edge weights. 
\par
Therefore to estimate the edge structure of a Gaussian graphical model, we consider the problem of estimating the precision matrix,
based on a sample of $n$ independent instances $X^1,\dots,X^n$, distributed  as $X^0$. 
We are not only interested in point estimation, but in quantifying the uncertainty of estimation such as constructing confidence intervals and tests for edge weights. 
Confidence intervals and tests can be used for identifying significant variables or testing whether networks corresponding to different populations are identical.
\par
The challenge arises due to the high-dimensional regime where the number of unknown parameters can be much larger than the number of observations $n$. 
It is instructive to first inspect  the low-dimensional setting. In the regime when $p$ is fixed and the observations are normally distributed with $\mathbb EX^i=0,i=1,\dots,n$, the sample covariance matrix 
$\hat\Sigma:=X^TX/n$ (where $X$ is the $n\times p$ matrix of observations $X^1,\dots,X^n$)
is the maximum likelihood estimator of the covariance matrix. The inverse of the sample covariance matrix $\hat\Theta:=\hat\Sigma^{-1}$ is the maximum likelihood estimator of the precision matrix. Asymptotic linearity of $\hat\Theta$ follows by the decomposition 
\begin{equation}
\label{deco0}
\hat\Theta-\Theta_0 = -\Theta_0 (\hat\Sigma-\Sigma_0)\Theta_0 + \text{rem}_0,
\end{equation}
where $\text{rem}_0:=-\Theta_0(\hat\Sigma-\Sigma_0)(\hat\Theta-\Theta_0)$ is the remainder term.
The term $\text{rem}_0$ can be bounded by H\"older's inequality
to obtain 
$$\|\text{rem}_0\|_\infty\leq \|\Theta_0(\hat\Sigma-\Sigma_0)\|_\infty \vertiii{\hat\Theta-\Theta_0}_1,$$
where we used the notation $\|A\|_\infty = \max_{1\leq i,j\leq p} |A_{ij}|$ for the supremum norm of a matrix $A$ and 
$\vertiii{A}_1:=\max_{1\leq j \leq p}\sum_{i=1}^p |A_{ij}|$  for the $\ell_1$-operator norm. 
If the  fourth moments of $X^{i}$'s are bounded, the decomposition \eqref{deco0} implies rates of convergence $\|\hat\Theta-\Theta_0\|_\infty=\mathcal O_P(1/\sqrt{n}),$ where $\mathcal O_P(1)$ denotes boundedness in probability. 
The remainder term then satisfies $\|\text{rem}_0\|_\infty =o_P(1/\sqrt{n}),$  $o_P(1)$ denoting convergence in probability to zero.
Therefore, $\hat\Theta$ is indeed an asymptotically linear estimator of $\Theta_0$ and in this sense, we can say it is asymptotically unbiased.
Moreover, $\hat\Theta$ is asymptotically normal with a limiting normal distribution.
\par 
In high-dimensional settings, the sample covariance matrix does not perform well
(see \cite{johnstone2} and \cite{jl}) and if $p>n$, it is singular with probability one. 
Various methods have been proposed that try to reduce the variance of the sample covariance matrix at the price of introducing some bias.
The idea of banding or thresholding the sample covariance matrix  was studied in \cite{bickel2}, \cite{bickel} and  \cite{elkaroui}. 
Methods inducing sparsity through Lasso regularization were studied by another stream of works. 
These can be divided into two categories: global methods and local (nodewise) methods. 
Global methods estimate the precision matrix typically via a regularized log-likelihood, while nodewise methods 
split the problem into a series of linear regressions by estimating neighbourhood of each node in the underlying graph.
A popular global method is the $\ell_1$-penalized maximum likelihood estimator, known as the graphical Lasso.
Its theoretical properties were studied in a number of papers, including \cite{yuan}, \cite{glasso}, \cite{rothman} and \cite{ravikumar}. 
The local approach on estimation of precision matrices in particular
includes the regression approach \cite{buhlmann},\cite{yuan2},\cite{cai},\cite{sunzhang} which uses a Lasso-type algorithm or Dantzig
selector \citep{candes2007} to estimate each column or a smaller part of
the precision matrix individually.
\par
Inference for parameters in high-dimensional undirected graphical models was studied in several papers.
Multiple testing for conditional dependence in Gaussian graphical models with asymptotic control of false discovery rates was considered in \cite{liu2013gaussian}. 
The work \cite{wasserman2014berry} proposes methodology for inference about edge weights based on Berry-Esseen bounds and the bootstrap for certain special classes of high-dimensional graphs.
Another line of work (\cite{zhou}, \cite{jvdgeer14} and \cite{jvdgeer15}) proposes asymptotically normal estimators for edge weights in Gaussian graphical models based on different modifications of initial Lasso-regularized estimators. In particular, the paper \cite{zhou} proposes nodewise regression where each {pair}  of variables, $(X_i,X_j)$, is regressed on all the remaining variables; this yields estimates for the parameters of the joint conditional distribution of $(X_i,X_j)$ given all the other variables. The papers \cite{jvdgeer14} and \cite{jvdgeer15} propose methodology inspired by the de-biasing approach in high-dimensional linear regression that was studied in \cite{zhang}, \cite{vdgeer13} and \cite{stanford1}. 
This chapter discusses and unifies the ideas from the papers \cite{jvdgeer14} and \cite{jvdgeer15}. 
\par
A different approach to structure learning in undirected graphical models is the Hyv\"arinen score matching (see e.g. \cite{drton2016structure} for a discussion of this approach). 
Methodology for asymptotically normal estimation of edge parameters in pairwise (not necessarily Gaussian) graphical models based on Hyv\"arinen scoring was proposed in \cite{NIPS2016_6530}.

\subsection{De-biasing regularized estimators}
\label{sec:debias}
The idea of using regularized estimators  for construction of asymptotically normal estimators is based on removing the bias associated with the penalty. 
Consider a real-valued loss function $\rho_{\Theta}$ and let $R_n(\Theta):=\sum_{i=1}^n \rho_{\Theta}(X^i)/n$ denote the average risk, given an independent sample $X^1,\dots,X^n.$
Under differentiability conditions, a regularized M-estimator $\hat\Theta$ based on the risk function $R_n$ can often be characterized by estimating equations
\begin{equation}
\label{ee}
\dot R_n({\hat\Theta_{}}) + \xi(\hat\Theta)=0,
\end{equation}
where $\dot R_n$ is the gradient of $R_n$ and $\xi(\hat\Theta)$ is a (sub-)gradient  corresponding to the regularization term, evaluated at $\hat\Theta$.
The idea is to improve on the initial estimator by finding a root $\hat\Theta_{\text{de-bias}}$ closer to the solution of estimating equations without the bias term $\xi(\hat\Theta)$, i.e. a new estimator $\hat\Theta_{\text{de-bias}}$ such that $\dot R_n(\hat\Theta_{\text{de-bias}})\approx 0$. A natural way is to define a corrected estimator $\hat\Theta_{\text{de-bias}}$ from a linear approximation to $\dot R_n$ 
\begin{equation}
\label{la}
\dot R_n({\hat\Theta_{}}) + \ddot R_n(\hat\Theta) (\hat\Theta_{\text{de-bias}}-\hat\Theta)=0.
\end{equation}
In high-dimensional settings, the matrix $ \ddot R_n(\hat\Theta) $ is typically rank deficient and thus not invertible. Suppose that we have an approximate inverse  denoted by $ \ddot R_n(\hat\Theta)^{\text{inv}}$. Then we can approximately solve \eqref{la} for $\hat\Theta_{\text{de-bias}}$
to obtain
\begin{equation}\label{de-bias}
\hat\Theta_{\text{de-bias}}\approx 
\hat\Theta -\ddot R_n({\hat\Theta_{}})^{\text{inv}} \dot R_n({\hat\Theta})
,\end{equation}
provided that the remainder term is small.
We refer to the step \eqref{de-bias} as the de-biasing step since the correction term is proportional to the bias term.
Generally speaking,  if the initial estimator $\hat\Theta$ and the approximate inverse of $\ddot R_n({\hat\Theta_{}})$ are  consistent in a strong-enough sense, then the new estimator $\hat\Theta_{\text{de-bias}}$ will be a consistent estimator of its population version $\Theta_0$ per entry at the parametric rate.
The de-biasing step \eqref{de-bias} may be viewed 
as one step of the Newton-Raphson scheme for numerical optimization. 
\par 
In consecutive sections, we will look  in detail at the bias of several particular examples of regularized estimators,
including the graphical Lasso (\cite{yuan}) and nodewise Lasso (\cite{buhlmann}). 
We now provide a unified de-biasing scheme  which covers both special cases treated below (see also \cite{sf}, Chapter 14). 
Suppose that a (possibly non-symmetric) estimator $\hat\Theta$ is available which is an approximate inverse of 
$\hat\Sigma$ in the sense that the following condition is satisfied
\begin{equation}\label{kkt1}
\hat\Sigma \hat\Theta - I + \eta(\hat\Theta)=0,
\end{equation}
where $\eta(\hat\Theta)$ is a bias term. This condition in some sense corresponds to the estimating equations \eqref{ee}. 
We can express $\hat\Theta$ from \eqref{kkt1} by straightforward algebra  which leads to the decomposition 
\begin{equation}
\label{deco.al}
\hat\Theta + \hat\Theta^T\eta(\hat\Theta) -\Theta_0 = -\Theta_0 (\hat\Sigma -\Sigma_0)\Theta_0
+\text{rem}_{0}+\text{rem}_\text{reg},
\end{equation}
where 
$$\text{rem}_{\text{reg}}=  (\hat\Theta-\Theta_0)^T \eta(\hat\Theta).$$
Compared to the regime with $p$ fixed, there is an additional remainder $\text{rem}_{\text{reg}}$ corresponding to the bias term. 
Provided that the remainder terms $\text{rem}_{0} $ and  $\text{rem}_{\text{reg}}$ are small enough, we can take as a new, de-biased estimator, 
$\hat T:=\hat\Theta + \hat\Theta^T \eta(\hat\Theta).$ The bias term $\eta(\hat\Theta)$ can be expressed from \eqref{kkt1} as $\eta(\hat\Theta)= -(\hat\Sigma \hat\Theta-I)$.
Hence we obtain
\begin{equation}\label{desp}
\hat T=  \hat\Theta + \hat\Theta^T - \hat\Theta^T \hat\Sigma \hat\Theta.
\end{equation}
Bounding the remainders $\text{rem}_{0}$ and $\text{rem}_{\text{reg}}$ in high-dimensional settings requires more refined arguments than when $p$ is fixed. Looking at the remainder $\text{rem}_{\text{reg}}$, we can again invoke H\"older's inequality to obtain
$$\|\text{rem}_{\text{reg}}\|_\infty= \|(\hat\Theta-\Theta_0)^T \eta(\hat\Theta)\|_\infty 
\leq 
\vertiii{\hat\Theta-\Theta_0}_1\|\eta(\hat\Theta)\|_\infty.$$
Thus it suffices to control the rates of convergence of $\hat\Theta$
in  $\vertiii{\cdot}_1$-norm and control the absolute size of entries of the bias term.
\par
Provided that the remainders are of small order $1/\sqrt{n}$ in probability, asymptotic normality per elements of $\hat T$ is a consequence of asymptotic linearity and can  be established under tail conditions on $X^i$'s, by applying the Lindeberg's central limit theorem.

\subsection{Graphical Lasso}

If the observations are independent $\mathcal N(0,\Sigma_0)$-distributed, 
the log-likelihood function is proportional to
$$\ell(\Theta):=\text{tr}(\hat\Sigma \Theta) - \log \text{det}(\Theta).$$
The graphical Lasso (see \cite{yuan}, \cite{aspremont}, \cite{glasso}) is based on the Gaussian log-likeli\-hood function but regularizes it via an $\ell_1$-norm penalty on the off-diagonal elements of the precision matrix. The diagonal elements of the precision matrix correspond to certain partial variances and thus should not be penalized. 
 The graphical Lasso is defined by
\begin{equation}\label{glasso}
\hat\Theta = \text{argmin}_{\Theta=\Theta^T, \Theta\succ 0} 
\text{tr}(\hat\Sigma \Theta) - \log \text{det}(\Theta)+ \lambda \|\Theta^-\|_1,
\end{equation}
where $\lambda$ is non-negative tuning parameter and we optimize over symmetric positive definite matrices, denoted by $\succ$. Here $\Theta^-$ represents the matrix obtained by setting the diagonal elements of $\Theta$ to zero and $\|\Theta^-\|_1$ is the $\ell_1$-norm of the vectorized version of $\Theta^-$. 
The usefulness of the graphical Lasso is not limited only to Gaussian settings; the theoretical results below show that it performs 
well as an estimator of the precision matrix in a large class of non-Gaussian settings. 
\par 
A disadvantage of the graphical Lasso \eqref{glasso} is that the penalization does not take into account 
that the variables have in general a different scaling.
To take these  differences in the variances into account, we define a modified graphical Lasso  
with a weighted penalty. To this end, let $\hat W^2:= \text{diag}(\hat\Sigma)$ be the diagonal matrix obtained from the diagonal of $\hat\Sigma$. We let
\begin{equation}
\label{glasso.w}
\hat\Theta_{\text{w}} = \text{argmin}_{\Theta=\Theta^T, \Theta\succ 0}  \text{tr}(\hat\Sigma \Theta) - \log \text{det}(\Theta)
+ 
\sum_{i\not =j} \hat W_{ii}\hat W_{jj}|\Theta_{ij}|.
\end{equation}
The weighted graphical Lasso $\hat\Theta_{\text{w}}$ is related to a graphical Lasso based on the sample correlation matrix $\hat R:=\hat W^{-1}\hat\Sigma \hat W^{-1}$. To clarify the connection, we define 
\begin{equation}
\label{glasso.n}\hat\Theta_{\text{norm}} = \text{argmin}_{\Theta=\Theta^T, \Theta\succ 0}  \text{tr}(\hat R \Theta) - \log \text{det}(\Theta)+ \|\Theta_{}^-\|_1.
\end{equation}
Then it holds that $\hat\Theta_{\text{w}} = \hat W^{-1} \hat\Theta_{\text{norm}}\hat W^{-1}.$ 
The estimator $\hat\Theta_{\text{norm}}$ is of independent interest, if the parameter of interest is the inverse correlation matrix rather than the precision matrix.
The estimators $\hat\Theta_\text{w}$ and $\hat\Theta_{\text{norm}}$ are also useful from a theoretical perspective as will be shown in the sequel.

\par
We now apply the de-biasing ideas of Section \ref{sec:debias} to the graphical Lasso estimators defined above, demonstrating the procedure on $\hat\Theta$.
By definition, the graphical Lasso is invertible, and the Karush-Kuhn-Tucker (KKT) conditions yield
$$\hat\Sigma - \hat\Theta^{-1} + \lambda \hat Z =0,$$
where 
\[
\hat Z_{ij}=
 \text{sign}(\hat\Theta_{ij})\;\; \text{if }\hat\Theta_{ij}\not =0,\quad 
\text{ and }\quad \|\hat Z\|_\infty \leq 1.
\]
Multiplying by $\hat\Theta,$ we obtain
$$\hat\Sigma\hat\Theta - I + \lambda \hat Z\hat\Theta=0.$$
 In line with Section \ref{sec:debias} above,  
this implies the decomposition 
\begin{equation*}
\hat\Theta + \hat\Theta^T\eta(\hat\Theta) -\Theta_0 = -\Theta_0 (\hat\Sigma -\Sigma_0)\Theta_0
+\text{rem}_{0}+\text{rem}_\text{reg},
\end{equation*}
with the bias term $\eta(\hat\Theta)=\lambda \hat Z\hat\Theta.$
 To control the remainder terms $\text{rem}_0$ and $\text{rem}_{\text{reg}}$, 
we need bounds for the $\ell_1$-error of $\hat\Theta$ and to control the bias term,  
it is sufficient to control the upper bound $\|\eta(\hat\Theta)\|_\infty 
= \|\lambda \hat Z\hat\Theta\|_\infty\leq \lambda \vertiii{\hat\Theta}_1$.

\subsubsection*{Oracle bounds}
\par
Oracle results for the graphical Lasso were studied in \cite{rothman} under sparsity conditions and mild regularity conditions. In \cite{ravikumar}, stronger results were derived under stronger regularity conditions (and weaker sparsity conditions). Here we revisit these results and provide several extensions. 
\par
We summarize the main theoretical conditions which require boundedness of the eigenvalues of the true precision matrix and certain tail conditions.
\begin{cond}[Bounded spectrum]
\label{eig}
The precision matrix $\Theta_0:=\Sigma_0^{-1}$ exists
and there exists a universal constant $L\geq 1$ such that 
$$1/L \leq \Lambda_{\min}(\Theta_0) \leq \Lambda_{\max}(\Theta_0) \leq L.$$
\end{cond}
\begin{cond}[Sub-Gaussianity]
\label{subgv}
The observations  $X^i,i=1,\dots,n,$ are uniformly sub-Gaussian vectors, i.e. 
there exists a universal constant $K>0$ such that for every $\alpha\in\mathbb R^p$, $\|\alpha\|_2 = 1$ it holds
$$
\mathbb E\exp\left({{|\alpha^T X^i|^2}/{K^2 }}\right) 
\leq 
2  \quad (i=1,\dots,n).$$

\end{cond}

Under \condi \ref{subgv}, the Bernstein inequality implies concentration results for $\hat\Sigma$ as formulated in Lemma \ref{conc1} below. The proof is omitted and may be found in \cite{hds} (Lemma 14.13). We denote the Euclidean norm by $\|\cdot\|_2$ and the $i$-th column of a matrix $A$ by $A_i$.
\begin{lemma}\label{conc1}
Assume \condi \ref{subgv}  and that for non-random matrices $A,B\in\mathbb R^{p\times p} $ it holds that $\|A_i\|_2\leq M$ and $ \|B_i\|_2\leq M$ for all $i=1,\dots,p$. 
Then for all $t>0$, with probability at least $1-e^{-nt}$ it holds that
$$\|A^T(\hat\Sigma -\Sigma_0)B\|_\infty/(2M^2K^2) \leq t + \sqrt{2t} +\sqrt{\frac{2\log (2p^2)}{n}}+\frac{\log (2p^2)}{n}.$$

\end{lemma}

\par
To derive oracle bounds for the graphical Lasso, we rely on certain sparsity conditions on the entries of the true precision matrix.
To this end, we define 
for $j=1,\dots,p$,
$$D_j:= \{(i,j):\Theta_{ij}^0\not = 0, i\not =j\}, \quad d_j:=\text{card}(D_j), \quad d:=\max_{j=1,\dots,p}|d_j|.$$ The quantity $d_j$ is then the degree of a node $X_j$ and $d$ corresponds to the maximum vertex degree in the associated graphical model (excluding vertex self-loops).
Furthermore, we define
$$S:=\bigcup_{j=1}^p D_j, \quad s:=\sum_{j=1}^p d_j,$$
thus $S$ denotes the overall off-diagonal sparsity pattern and $s$ is the overall number of edges (excluding self-loops). 
\par 
The following theorem is an extension of the result for the graphical Lasso in \cite{rothman} to the $\ell_1$-norm. The paper \cite{rothman}  derives rates in Frobenius norm $\|\cdot\|_F$, which is defined as $\|A\|_F^2:=\sum_{i,j} |A^{2}_{ij}|$ for a matrix $A$. 
\begin{theorem}[Regime $p\ll n$
]
\label{glasso.rate}
Let $\hat\Theta$ be the minimizer defined by \eqref{glasso}.
Assume \conds \ref{eig} and \ref{subgv}.
Then for $\lambda$ satisfying $2\lambda_0\leq \lambda\leq 1/(8L c_L)$ and $ 8c_L^2 s\lambda^2 + 8 c_L p\lambda_0^2\leq \lambda_0/(2L),$
 on the set $\|\hat\Sigma-\Sigma_0\|_\infty \leq \lambda_0$, it holds that 
\begin{eqnarray*}
\|\hat\Theta_{}-\Theta_0\|_F^2 /c_L
+\lambda \|\hat\Theta_{}^--\Theta_0^-\|_1  
\leq 8c_L^2 s\lambda^2 + 8 c_L p\lambda_0^2,
\end{eqnarray*}
and
$$\vertiii{\hat\Theta_{}-\Theta_0}_1 \leq 
16c_L^2 (p+s)\lambda
,$$
where $c_L=8L^2.$
\end{theorem}
The slow rate in the result above arises from the part of the estimation error $\text{tr}[(\hat\Sigma-\Sigma_0)(\hat\Theta-\Theta_0)]$  which is related to the diagonal elements of the precision matrix. 
However, proper normalizing removes this part of the estimation error. 
\par
The following theorem  derives an extension of
\cite{rothman} for the normalized graphical Lasso $\hat\Theta_{\text{norm}}$. 
Denote the true correlation matrix by 
$R_0$ and let $K_0:=R_0^{-1}$ denote the inverse correlation matrix.

\begin{theorem}[Regime $p\gg n$]
\label{norm.glasso.rate}
Assume that \conds  \ref{eig} and  \ref{subgv} hold. 
Then for $\lambda$ satisfying $2\lambda_0\leq \lambda\leq 1/(8L^2)$ and $ 8c_L^2 s\lambda^2 \leq \lambda_0/(2L),$ 
on the set $\|\hat R-R_0\|_\infty \leq \lambda_0$ it holds for some constant $C_L>0$ that
\begin{eqnarray*}
\|\hat\Theta_{\emph{norm}}-K_0\|_F^2 
+\lambda \|\hat\Theta_{\emph{norm}}^--K_0^-\|_1  \leq 8c_L^2 s\lambda^2
,
\end{eqnarray*}
$$\vertiii{\hat\Theta_{\emph{norm}}-K_0}_1 \leq 8c_Ls\lambda^2 + 8c_L^2 s\lambda.$$
$$\vertiii{\hat\Theta_{\emph{w}}-\Theta_0}_1 \leq C_L s\lambda,$$
where $c_L=8L^2.$

\end{theorem}
Using the normalized graphical Lasso leads to  faster  rates in Frobenius norm and $\ell_1$-norm as shown above. The rates for $\hat\Theta_\text{w}$ in $\vertiii{\cdot}_1$-norm can be then established immediately. 
To derive a high-proba\-bility bound for $\|\hat R -R_0\|_\infty$, we may apply Lemma \ref{conc1} together with H\"older's inequality
 to obtain $\|\hat R -R_0\|_\infty = \mathcal O_P(\sqrt{\log p/n})$. Hence, 
$\vertiii{\hat\Theta_{\text{w}}-\Theta_0}_1 =\mathcal O_P(s\sqrt{\log p/n}).$
\begin{remark}\normalfont
The above result requires a strong condition on the sparsity in $\Theta_0$, i.e. there can only be very few non-zero coefficients due to the restriction $ 8c_L^2 s\lambda^2 \leq \lambda_0/(2L).$ This condition guarantees that a margin condition is satisfied. 
An example of a graph satisfying the condition is a star graph with order $\sqrt{n}$ edges.
\end{remark}

\subsubsection*{Asymptotic normality}
Once oracle results in $\ell_1$-norm are available, we can easily obtain results on asymptotic normality of the de-biased estimator $2\hat\Theta-\hat\Theta\hat\Sigma\hat\Theta$ for the graphical Lasso and its weighted version.
We denote the asymptotic variance of the de-biased estimator by $\sigma_{ij}^2:=n\text{var}((\Theta_i^0)^T\hat\Sigma\Theta_j^0)$. The arrow $\rightsquigarrow $ denotes convergence in distribution and for a matrix $A$ we denote $(A)_{ij}$ its $(i,j)$-entry.
\begin{theorem}[Regime $p\ll n$
]
\label{glasso.rem}
Assume \conds \ref{eig}, \ref{subgv}, $\lambda\asymp \sqrt{\log p/n}$ and that  $(p+s)\sqrt{d}=o(\sqrt{n}/\log p)$. Then it holds
that
\begin{equation}\label{an}
2\hat\Theta - \hat\Theta  \hat\Sigma\hat\Theta -\Theta_0= 
-\Theta_0 (\hat\Sigma - \Sigma_0)\Theta_0 + \emph{rem},
\end{equation}
where 
$$\|\emph{rem}\|_\infty =\mathcal O_P\left(24 (8L^2)^2 L (p+s)\sqrt{d+1}\lambda^2\right) = o_P(1/\sqrt{n}).
$$
Moreover, for $i,j=1,\dots,p,$
$$\sqrt{n}(2\hat\Theta_{} - \hat\Theta \hat\Sigma \hat\Theta-\Theta_{0})_{ij}/\sigma_{ij} \rightsquigarrow \mathcal N(0,1).$$
\end{theorem}
The result of Theorem \ref{glasso.rem} gives us tools to construct approximate confidence intervals and tests for individual entries of $\Theta_0.$ However, we need a consistent estimator of the asymptotic variance $\sigma_{ij}.$ 
For the Gaussian case, one may take
$\hat\sigma_{ij}^2 := \hat\Theta_{ii}\hat\Theta_{jj} + \hat\Theta_{ij}^2.$ 
We omit the proof of consistency of $\hat\sigma_{ij}$ and point the reader to \cite{jvdgeer15}, where other distributions are treated as well.
Moreover, Theorem \ref{glasso.rem} implies parametric rates of convergence for estimation of individual entries and a rate of order $\sqrt{\log p/n} $ for the error in supremum norm.
Theorem \ref{glasso.rem} requires a stronger sparsity condition that the corresponding oracle-type inequality in Theorem 
\ref{glasso.rate}. This is to be expected as will be argued in Section \ref{sec:con}. 

\par
Using the weighted graphical Lasso, 
the results of Theorem \ref{glasso.rem} can be established under weaker conditions as shown in the following theorem.
\begin{theorem}[Regime $p\gg n$]
\label{denorm.glasso.rem}
Assume \conds \ref{eig}, \ref{subgv}  and $s^{}\sqrt{d}=o(\sqrt{n}/\log p).$ Then for $\lambda\asymp \sqrt{\log p/n},$ 
the asymptotic linearity \eqref{an} holds with $\hat\Theta_{\emph{w}}$,  
where 
$$\|\emph{rem}\|_\infty =\mathcal O_P\left( 12 (8L^2)^2s^{}\sqrt{d+1}\lambda^2\right)=o_P(1/\sqrt{n}).$$
Moreover, for $i,j=1,\dots,p$,
$\sqrt{n}(2\hat\Theta_{\emph{w}}+ \hat\Theta_{\emph{w}} \hat\Sigma \hat\Theta_{\emph{w}}-\Theta_{0})_{ij}/\sigma_{ij} \rightsquigarrow \mathcal N(0,1).$

\end{theorem}

\par
If the parameter of interest is instead  the inverse correlation matrix, we formulate a partial result below.  

\begin{proposition}[Regime $p\gg n$]
\label{norm.glasso.rem}
Assume \conds \ref{eig}, \ref{subgv}, $\lambda\asymp \sqrt{\log p/n}$ and that $s^{}\sqrt{d}=o(\sqrt{n}/\log p).$ Then it holds 
$$2\hat\Theta_{\emph{norm}} - \hat\Theta_{\emph{norm}}  \hat R \hat\Theta_{\emph{norm}} -K_0
= 
-K_0 (\hat R - R_0)K_0 + \emph{rem},
$$
$$\|\emph{rem}\|_\infty =\mathcal O_P\left(12 (8L^2)^2 Ls^{}\sqrt{d+1}\lambda^2\right)=o_P(1/\sqrt{n}).$$
\end{proposition}
To claim asymptotic normality of $K_0 (\hat R - R_0)K_0$ per entry would require extensions of central limit theorems to high-dimensional settings (see \cite{cltvch}) and an extension of the $\delta$-method. We do not study these extensions in the present work. To give a glimpse, in the regime when $p$ is fixed, by the central limit theorem it follows  that 
 $\sqrt{n}(\hat\Sigma-\Sigma_0)\rightsquigarrow \mathcal N_{p^2}(0,C)$, where $C$ is the asymptotic covariance matrix.
Then we may apply the $\delta$-method with 
$h_{ij}(\Sigma):=(K_i^0)^T \text{diag}(\Sigma)^{-1/2}\Sigma \text{diag}(\Sigma)^{-1/2}K_j^0$ to obtain 
 $\sqrt{n}(h_{ij}(\hat\Sigma)-h_{ij}(\Sigma_0)) \rightsquigarrow \mathcal N (0, \dot h(\Sigma_0)^T C \dot h(\Sigma_0)).$
\par

Finally, we show that the sparsity conditions in the above results may be further relaxed under a stronger regularity condition on the true precision matrix. 
 We provide here a simplified version of the result in \cite{jvdgeer14} which assumes an irrepresentability condition on the true precision matrix. 
Let $\kappa_{\Sigma_0}$ be the $\ell_\infty$-operator norm of the true covariance matrix $\Sigma_0$, i.e. 
$\kappa_{\Sigma_0} = \vertiii{\Sigma_0}_1$. 
Let $ H_0$ be the Hessian of the expected  Gaussian log-likelihood evaluated at $\Theta_0$, i.e. $H_0=\Sigma_0\otimes \Sigma_0.$ 
When calculating the Hessian matrix, we treat the precision matrix as non-symmetric; this will allow us to accommodate non-symmetric estimators as well.
For any two subsets $T$ and $T'$ of $\mathcal V\times \mathcal V$, we use $H^0_{TT'}$ to denote the $|T|\times|T'|$ matrix with rows 
and columns of $H_0$ indexed by $T$ and $T'$ respectively.
Define 
$\kappa_{H_0}=\vertiii{(H^0_{SS})^{-1}}_1.$
\begin{cond}
(Irrepresentability condition)\label{ir}
There exists $\alpha\in (0,1]$ such that
$\max_{e \in S^c}\|H^0_{eS}(H^0_{SS})^{-1}\|_1 \leq 1-\alpha,$ where $S^c $ is the complement of $S.$
\end{cond}
Condition \ref{ir} is an analogy of the irrepresentable condition for variable selection in linear regression (see \cite{vdgeer09}). 
If we define the zero-mean edge random variables as
$Y_{(i,j)}:= X_i X_j - \mathbb E(X_iX_j),$
then the matrix $H_0$ corresponds to covariances of the edge variables, in particular
$H^0_{(i,j),(k,l)} + H^0_{(j,i),(k,l)} = \text{cov}(Y_{(i,j)},Y_{(k,l)})$.
Condition \ref{ir} means that we require that no edge variable $Y_{(j,k)}$, which is not included in the edge set $S$, is highly correlated with variables in the edge set (see \cite{ravikumar}). The parameter $\alpha$ then is a measure of this correlation with the correlation growing when $\alpha\rightarrow 0$. Some examples of matrices satisfying the irrepresentable condition may be found in \cite{jvdgeer15}.

\begin{theorem}[Regime $p\gg n$
]
\label{glasso.rem.irr}
Assume that \conds \ref{eig}, \ref{subgv} and \ref{ir} are satisfied,  $d=o(\sqrt{n}/\log p),$ $\kappa_{\Sigma_0}=\mathcal O(1)$ 
and $\kappa_{H_0}=\mathcal O(1).$ Then for $\lambda\asymp \sqrt{\log p/n}$,
the asymptotic linearity \eqref{an} holds with $\hat\Theta_{}$,  
where 
$\|\emph{rem}_{}\|_\infty = \mathcal O_P(d\log p/n)= o_P(1/\sqrt{n}).$
Moreover,
$$\sqrt{n}(2\hat\Theta_{} - \hat\Theta \hat\Sigma \hat\Theta-\Theta_{0})_{ij}/\sigma_{ij} \rightsquigarrow \mathcal N(0,1).$$

\end{theorem}

The proof of Theorem \ref{glasso.rem.irr} may be found in \cite{jvdgeer14}.
We remark that under the irrepresentability condition, one can show that 
$|\hat\Theta_{ij}-\Theta^0_{ij}|=\mathcal O_P(1/\sqrt{n})$ (see \cite{ravikumar}). 
This means that one could use the graphical Lasso itself to construct confidence intervals of asymptotically optimal (parametric) size. However, this holds under the strong irrepresentability condition which is often violated in practice.
\par
Comparing the results obtained for the (weighted) graphical Lasso,
we see that the strongest result was attained under the irrepresentable condition and under the sparsity condition $d=o(\sqrt{n}/\log p)$. An analogous result has not yet been obtained for the graphical Lasso without the irrepresentable condition (under the same sparsity condition).
However, without the irrepresentable condition, we showed the same results for the weighted graphical Lasso under the sparsity condition $s^{}\sqrt{d}=o(\sqrt{n}/\log p)$. 
In the next section, we will consider a procedure based on pseudo-likelihood, for which we can derive identical results under weaker conditions, namely under the sparsity condition $d=o(\sqrt{n}/\log p)$ and under the \conds \ref{eig} and \ref{subgv}.

\subsection{Nodewise square-root Lasso}

\label{sec:node}
An alternative approach to estimate the precision matrix is based on linear regression.
The idea of nodewise Lasso is to estimate each column of the precision matrix by doing a projection of every column of the design matrix on all the remaining columns. While this is a pseudo-likelihood method, the decoupling into linear regressions gains more flexibility in estimating the individual scaling levels compared to the graphical Lasso which aims to estimate all the parameters simultaneously. 
Moreover, by splitting the problem up into a series of linear regressions, the computational burden is reduced compared to the graphical Lasso.
\par
In low-dimensional settings, regressing each column of the design matrix on all the other columns would simply recover the inverse of the sample covariance matrix $\hat\Sigma.$ However, due to the high-dimensionality of our setting, the matrix $\hat\Sigma$ is not invertible and we can only do approximate projections. If we assume sparsity in the precision matrix (and thus also in the partial correlations), this idea can be effectively carried out using the square-root Lasso (\cite{sqrtlasso}). \par
The theoretical motivation can be understood in greater detail from the population version of the method. For each $j=1,\dots,p,$ we define  
the vector of partial correlations $ \gamma_j^0=\{\gamma^0_{j,k},k\not = j\}$ as follows
\begin{equation}\label{true}
\gamma_j^0 := \text{argmin}_{\gamma\in\mathbb R^{p-1}} \mathbb E\|X_j - {} X_{-j}\gamma\|_2^2/n,
\end{equation}
and we denote the noise level by $\tau_j^2 = \mathbb E\|X_j - {} X_{-j}\gamma_j^0\|_2^2/n.$
Then one may  show that the $j$-th column of $\Theta_0$ can be recovered from the partial correlations $\gamma_j^0$ and the noise level $\tau_j$ using the following identity:
$\Theta^0_j = (-\gamma_{j,1},\dots, -\gamma_{j,j-1},1, -\gamma_{j,j+1} ,\dots,$ $-\gamma_{j,p})^T/\tau_j^2$.

  Hence, the idea is to define for each $j=1,\dots,p$ the estimators of the partial correlations, $\hat\gamma_j =\{\hat\gamma_{j,k},k=1,\dots,p,j\not = k\} \in \mathbb R^{p-1}$ using, for instance, the square-root Lasso with weighted penalty,

\begin{equation}\label{NR2}
\hat\gamma_j := \text{argmin}_{\gamma\in\mathbb R^{p-1}}
 \|X_j - {} X_{-j}\gamma\|_2/n+2\lambda \|\hat W_{-j}\gamma\|_1,
\end{equation}
where by $A_{-j}$ we denote the matrix $A$ without its $j$-th column.
We further define 
estimators of the noise level 
$$\hat\tau_j^2 := \|X_j - {} X_{-j}\hat\gamma_j\|_2^2/n,\quad \tilde\tau_j^2 :=\hat\tau_j^2+
\lambda\hat\tau_j\|\hat\gamma_j\|_1
,$$
for $j=1,\dots,p$.
Finally, we define 
the nodewise square-root Lasso estimator 
\begin{equation}
\label{nodew}
\hat\Theta_{\text{}} :=
\left( 
\begin{array}{cccc}
\;\;\;\;\;\;\;1/\tilde\tau_1^2 & -\tilde\gamma_{1,2}/\tilde\tau_1^2 & 
 \dots & -\tilde\gamma_{1,p}/\tilde\tau_1^2 \\
-\tilde\gamma_{2,1}/\tilde\tau_2^2 & \;\;\;\;\;\;\;1/\tilde\tau_2^2 & \dots &- \tilde\gamma_{2,p}/\tilde\tau_2^2 \\ 
\vdots &  \vdots & \ddots & \vdots \\
-\tilde\gamma_{p,1}/\tilde\tau_p^2 &  \dots & -\tilde\gamma_{p,p-1}/\tilde\tau_p^2 & \;\;\;\;\;\;1/\tilde\tau_p^2
\end{array}
\right)
\end{equation}
An equivalent way of formulating the definitions above is 
\begin{equation}\label{NRalt}
(\hat\gamma_j ,\hat\tau_j)= \text{argmin}_{\gamma\in\mathbb R^{p-1},\tau\in\mathbb R}
 \|X_j - {} X_{-j}\gamma\|_2^2/n/(2\tau) + \tau/2+2\lambda \|\hat W_{-j}\gamma\|_1.
\end{equation}
Alternative versions of the above estimator were considered in the literature. 
One may use the Lasso (\cite{lasso}) instead of the square-root Lasso (as in \cite{jvdgeer15}) or the Dantzig selector (see \cite{vdgeer14}).
Furthermore, one may define the nodewise square-root Lasso with $\hat\tau_j$ in place of $\tilde\tau_j.$ 
\par
The properties of the column estimator $\hat \Theta_j$ were studied in several papers (following \cite{buhlmann}) and it has been shown to enjoy oracle properties under the \conds \ref{eig}, \ref{subgv} and 
under the sparsity condition $d=\mathcal O(n/\log p)$ (see \cite{vdgeer13}, where a similar version was considered).
\par
In line with Section \ref{sec:debias}, we consider a de-biased version of the nodewise square-root Lasso estimator. 
The KKT conditions for the optimization problem (\ref{NR2}) give
\begin{equation}\label{kkt}
-{\hat\tau_j} X_{-j}^T(X_j-{} X_{-j}\hat\gamma_j)/n + \lambda \hat \kappa_j = 0,
\end{equation}
for $j=1,\dots,p$, where $\hat\kappa_j$ is an element of the sub-differential of the function $\gamma_j \mapsto \|\gamma_j\|_1$ at $\hat\gamma_j,$ i.e.
for $k\in \{1,\dots,p\}\setminus \{j\}$,
\[ 
\hat\kappa_{j,k} = 
 \text{sign}(\hat\gamma_{j,k}) \;\; \text{if } \hat\gamma_{j,k}\not = 0,\quad
\text{ and }\quad \|\hat\kappa_j\|_\infty\leq 1.
\]
If we define $\hat Z_j$ to be a $p\times 1$ vector 
$$\hat {Z}_j:= ( \hat \kappa_{j,1},\dots,\hat \kappa_{j,j-1},0, \hat \kappa_{j,j+1},\dots,\hat \kappa_{j,p} ),$$
then the KKT conditions may be equivalently stated as follows
$$\hat{\Sigma} \hat{\Theta}_j -  e_j - \lambda\frac{\hat\tau_j}{\tilde\tau_j^2 }\hat{ Z}_j=0.$$
\noindent
\noindent
Let $\hat Z$ be a matrix with columns $\hat Z_j$ for $j=1,\dots,p$, 
 $\hat\tau$ be a diagonal matrix with elements $\text{}(\hat\tau_1,\dots,\hat\tau_p)$ and similarly $\tilde\tau:=\text{diag}(\tilde\tau_1,\dots,\tilde\tau_p)$. 
As in Section \ref{sec:debias}, this yields the decomposition \eqref{deco.al} with a bias term 
$\eta(\hat\Theta):=\hat Z\Lambda \hat \tau \tilde\tau^{-2}.$ The bias term can then be controlled as  
$$\|\eta(\hat\Theta)\|_\infty\leq \lambda\|\hat\tau\|_\infty\|\tilde\tau^{-2}\|_\infty \|\hat Z\|_\infty \leq \lambda\max_{1 \leq j\leq p} \hat\tau_j/\tilde\tau_j^2.$$

\begin{theorem}[Regime $p\gg n$] 
\label{node.rem} 
Suppose that \conds \ref{eig}, \ref{subgv} are satisfied and $d=o(\sqrt{n}/\log p)$.
Let $\hat\Theta_{{\emph{node}}}$ be the estimator defined in \eqref{nodew} and let $\lambda\asymp\sqrt{\log p/n}$.
Then it holds
\begin{equation*}
\hat \Theta_{\emph{}}+\hat \Theta_{\emph{}}^T -\hat \Theta_{\emph{}}^T\hat\Sigma \hat \Theta_{\emph{}}-\Theta_{0} = 
- \Theta_{0}(\hat\Sigma -\Sigma_0)\Theta_{0} + {\emph{rem}},
\end{equation*}
where 
$\|\emph{rem}\|_\infty =\mathcal O_P( d\lambda^2)=o_P(1/\sqrt{n})$. 
Moreover,  
$$\sqrt{n}(\hat\Theta_{}+\hat\Theta_{}^T - \hat\Theta^T \hat\Sigma \hat\Theta-\Theta_{0})_{ij}/\sigma_{ij} \rightsquigarrow \mathcal N(0,1).$$
\end{theorem}

When the parameter of interest is the inverse correlation matrix, we can use the normalized version of the nodewise square-root Lasso and we obtain an analogous result.
\begin{proposition}[Regime $p\gg n$] 
\label{node.rem.norm} 
Suppose that \conds \ref{eig}, \ref{subgv} are satisfied, let $\lambda\asymp\sqrt{\log p/n}$ and $d=o(\sqrt{n}/\log p)$.
Then
\begin{equation*}
\hat \Theta_{\emph{norm}}+\hat \Theta_{\emph{norm}}^T -\hat \Theta_{\emph{norm}}^T\hat R \hat \Theta_{\emph{norm}}-\Theta_{0} = 
- K_{0}(\hat R -R_0)K_{0} + {\emph{rem}},
\end{equation*}
where 
$\|\emph{rem}\|_\infty =o_P(1/\sqrt{n}).$
\end{proposition}

\subsection{Computational view}
\label{sec:comput.view}

For the nodewise square-root Lasso, we need to solve $p$ square-root Lasso regressions, which can be efficiently handled using interior-point methods with polynomial computational time or first-order methods (see \cite{sqrtlasso}). 
Alternatively to nodewise square-root Lasso, the nodewise Lasso studied in \cite{jvdgeer15} may be used, which requires selection of a tuning parameter for each of the $p$ regressions. This can be achieved e.g. by cross-validation and can be implemented efficiently using parallel methods (\cite{lars}).
 The graphical Lasso presents a more computationally challenging problem; we refer the reader to e.g. \cite{glasso.comp}.
The computation of the de-biased estimator itself only involves simple matrix addition and multiplication.

\subsection{Simulation results}

We consider a setting with $n$ independent observations generated from $\mathcal N_p(0,\Theta_0^{-1}),$ where the precision matrix
$\Theta_0$ follows one of the three models:
\begin{enumerate}
\item
Model 1: $\Theta_0$ has two blocks of equal size and each block is a five-diagonal matrix with elements $(1,0.5,0.4) $ and $(2,1,0.6),$
respectively.
\item
Model 2: $\Theta_0$ is a sparse precision matrix  generated using the R package GGMselect (using the function simulateGraph() with parameter $0.07$). The matrix was converted to a correlation matrix with the function cov2cor().
\item
Model 3: $\Theta^0_{ij} =0.5^{|i-j|}$, $i,j=1,\dots,p.$
\end{enumerate}
We consider 6 different methods: the de-biased estimator based on the
\begin{enumerate}[(1)]
\itemsep-0.1em 
\item
  graphical Lasso (glasso)
\item
weighted graphical Lasso  (glasso-weigh), 
\item
nodewise square-root Lasso (node-sqrt) as defined above,
\item
 nodewise square-root Lasso with alternative $\tilde\tau$ as in  \cite{sunzhang} (node-sqrt-tau) 
\item
 nodewise Lasso as in \cite{jvdgeer15} (node)
\end{enumerate}
and we also consider
\begin{enumerate}[(1)]
\setcounter{enumi}{5}
\item
the maximum likelihood estimator (MLE).
\end{enumerate}
Furthermore, as a benchmark we report the oracle estimator (oracle) which applies maximum likelihood using the knowledge of true zeros in the precision matrix. We also report the target coverage and the efficient asymptotic variance of asymptotically regular estimators (see \cite{ae}) (perfect).

\begin{figure}[h]
\centering
\caption[Visualization of graphical models used in simulations.]{Visualization of graphical models used in simulations. Models 1,2 and 3 from left to right. For Model 3, we only plot edges with a weight greater than $0.1$.}
\vskip 0.3cm
\includegraphics[width=0.27\textwidth]{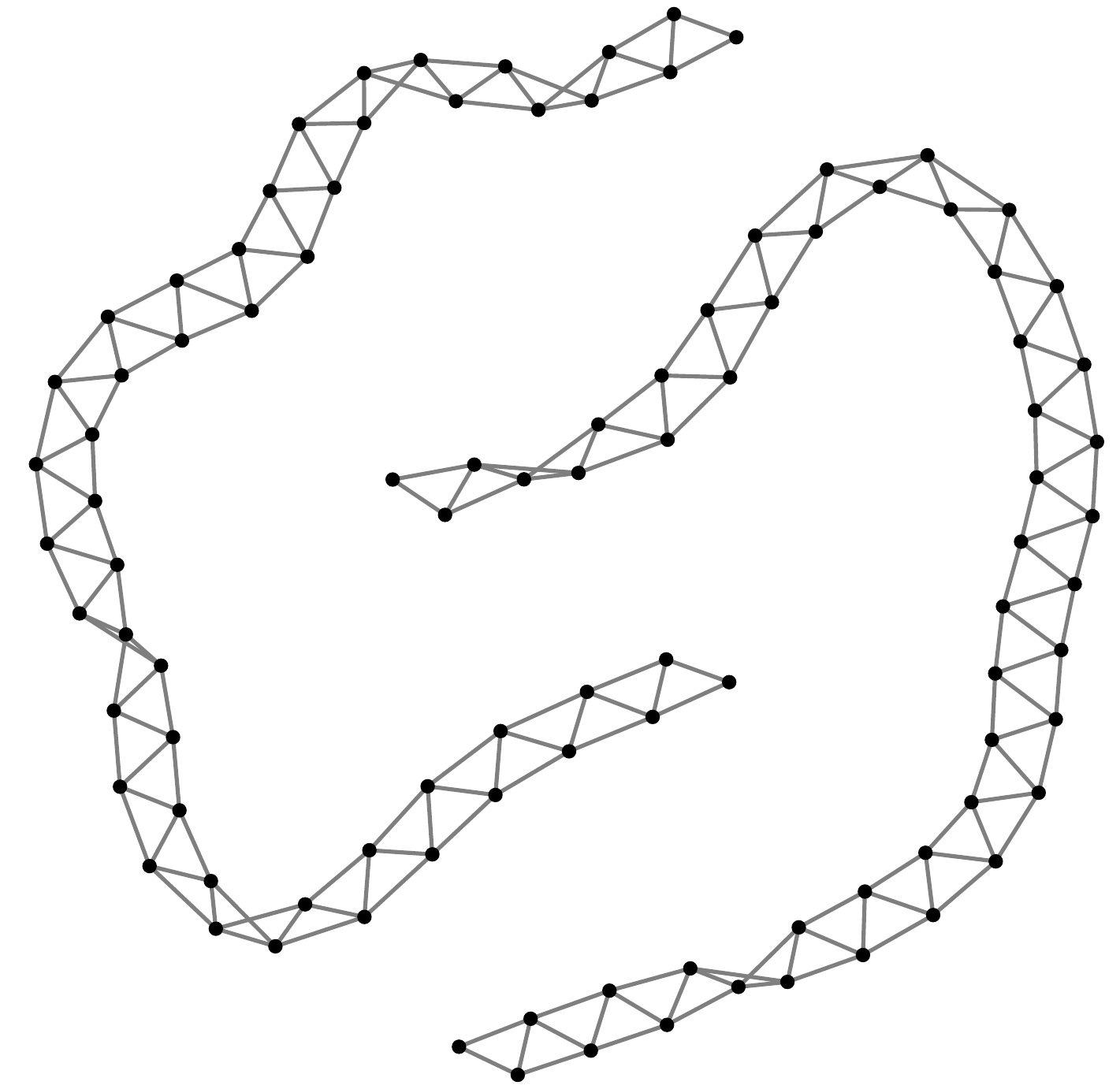}
\hskip 0.5cm
\includegraphics[width=0.27\textwidth]{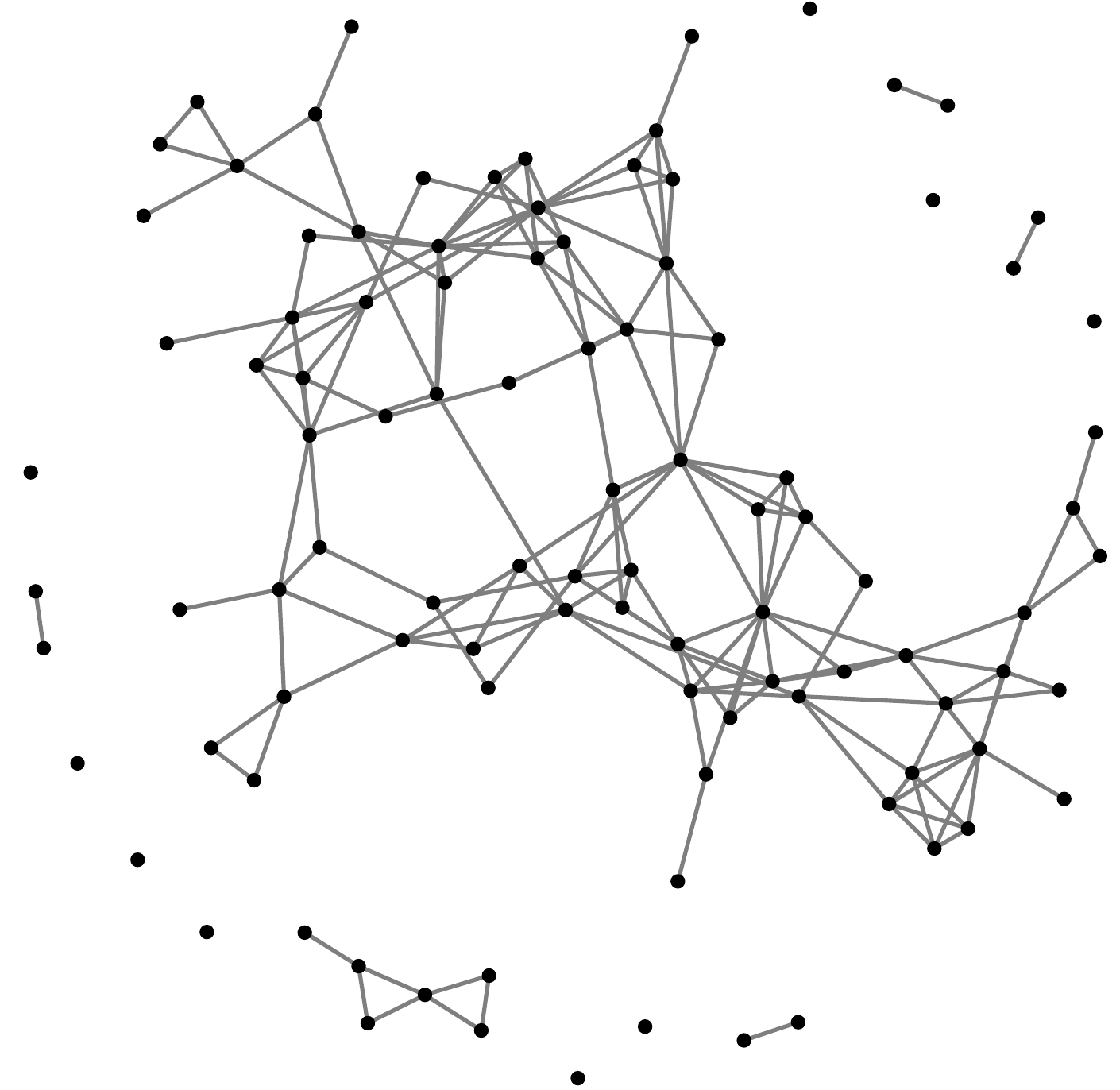}
\hskip 0.5cm
\includegraphics[width=0.27\textwidth]{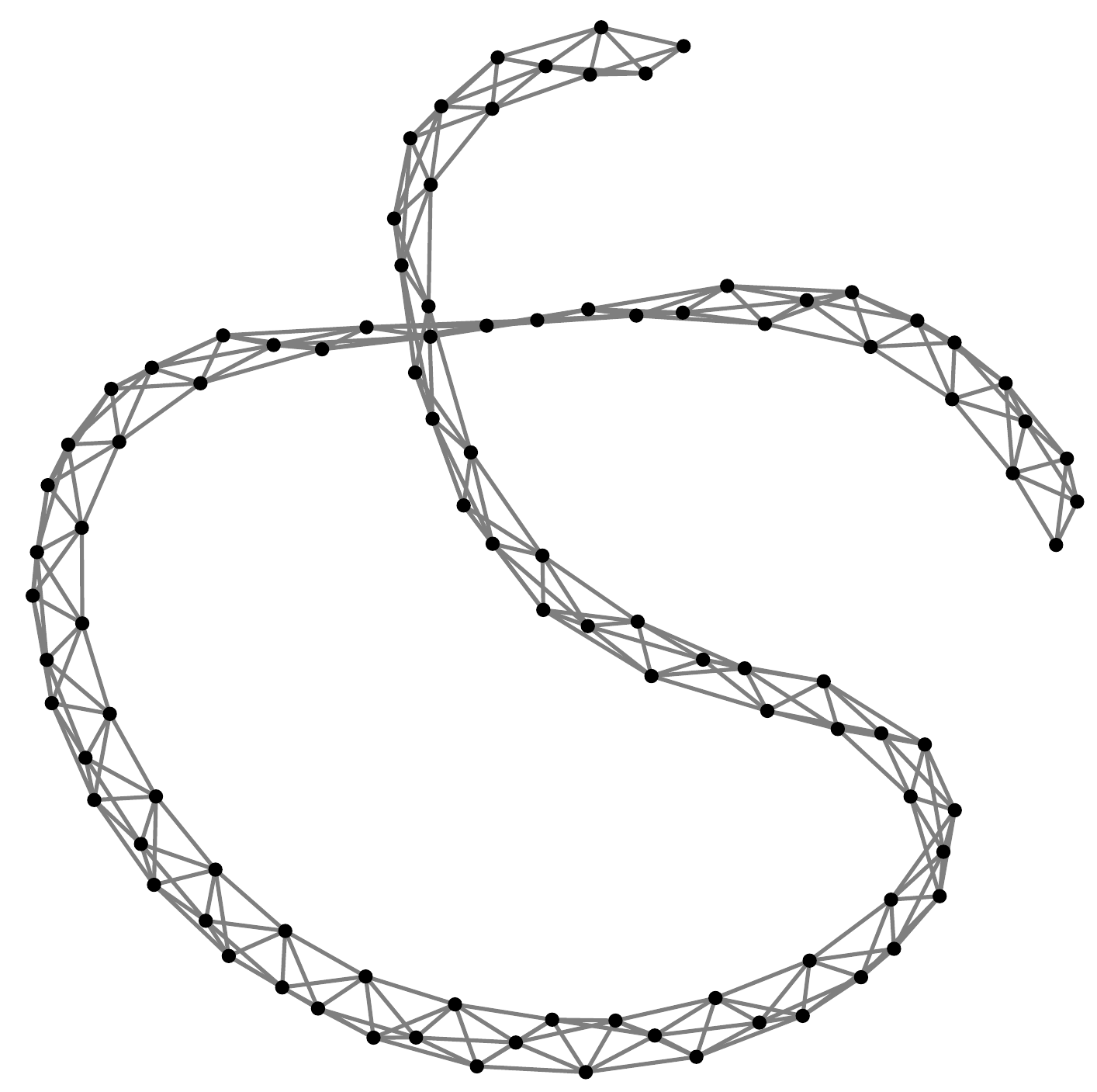}

\end{figure}

For the graphical Lasso (1) and the weighted graphical Lasso (2) we choose the tuning parameter by maximizing the likelihood
on a validation data set (a new data of the size $n$). 
For methods (3), (4) and (5), the universal choice $\sqrt{\log p/n}$ is used.
\par
We display results on confidence intervals for nominal coverage $95\%$ and for normally distributed observations in Tables \ref{table:m1},
\ref{table:m2}, \ref{table:m3} and \ref{table:m4}. For other nominal coverages, we obtain similar performance (these results are not reported). For other than Gaussian distributions, we refer the reader to the simulation results in \cite{jvdgeer15}.
Firstly, the results of the simulations suggest that the de-biased estimators perform significantly better than
the maximum likelihood estimator even though $p<n$ and secondly, the nodewise methods seem to outperform the graphical lasso methods in our settings.

\begin{table}
\caption[Average coverages and lengths of confidence intervals over the active set $S_0:=S\cup \{1,\dots,p\}$ and its complement $S_0^c$, over $100$ realizations.]{
Average coverages and lengths of confidence intervals over the active set $S_0:=S\cup \{1,\dots,p\}$ and its complement $S_0^c$, over $100$ realizations. The average value of the tuning parameters is reported in the last column. The benchmark ``estimators'' are labeled by a star.
The significance level is $0.05.$
}
\centering
\begin{tabular}{rlccccc}
\multicolumn{7}{c}{Model 1: Block 1: $(1,0.5,0.4)$, Block 2: $(2,1,0.6)$}  \\[0.1cm]
\multicolumn{7}{c}{$p=100,n=200$}  \\[0.1cm]
  \hline
	&     & \multicolumn{2}{c}{Coverage}& \multicolumn{2}{c}{Length}& Average $\lambda$\\ 
 & Method & $S_0$ & $S_0^c$ & $S_0$ & $S_0^c$ &\\ 
  \hline
  1 & glasso & 77.19 & 98.07 & 0.36 & 0.32 &  0.088\\ 
  2 & glasso-weigh & 35.02 & 98.65 & 0.31 & 0.27 & 0.088\\ 
  3 & node-sqrt & 89.92 & 94.02 & 0.48 & 0.42 &0.152\\ 
  4 & node-sqrt-tau & 83.48 & 97.40 & 0.38 & 0.33 &0.152\\ 
  5 & node & 90.58 & 96.77 & 0.41 & 0.35 & 0.152\\ 
  6 & MLE & 20.92 & 84.27 & 0.97 & 0.81 &-\\\hline 
  7 & oracle* & 94.95 & - & 0.49 & 0.40 &- \\ 
  8 & perfect* & 95.00 & 95.00 & 0.48 & 0.40&- \\ 
\end{tabular}
\vskip 0.6cm
\centering
\begin{tabular}{rlccccc}
\multicolumn{7}{c}{$p=100,n=400$}  \\[0.1cm]
  \hline
	&     & \multicolumn{2}{c}{Coverage}& \multicolumn{2}{c}{Length} & Average $\lambda$\\ 
 & Method & $S_0$ & $S_0^c$ & $S_0$ & $S_0^c$ &\\ 
  \hline
  1 & glasso & 84.28 & 97.53  & 0.27 & 0.23 & 0.049\\
  2 & glasso-weigh & 46.22 & 98.41 & 0.24 & 0.20 & 0.049\\
  3 & node-sqrt & 91.57 & 94.40 & 0.34 & 0.29 & 0.107\\
  4 & node-sqrt-tau & 87.11 & 97.13 & 0.28 & 0.24 & 0.107\\
  5 & node & 91.48 & 96.40 & 0.30 & 0.25 & 0.107\\
  6 & MLE & 41.42 & 91.29 & 0.41 & 0.38 & -\\\hline
  7 & oracle* & 94.87 & - & 0.34 & 0.29 & -\\
  8 & perfect* & 95.00 & 95.00 & 0.34 & 0.29 &-\\
\end{tabular}
\label{table:m1}
\end{table}

\begin{table}
\centering
\caption[Average coverages and lengths of confidence intervals over the active set $S_0:=S\cup \{1,\dots,p\}$ and its complement $S_0^c$, over $100$ realizations. ]{
Average coverages and lengths of confidence intervals over the active set $S_0:=S\cup \{1,\dots,p\}$ and its complement $S_0^c$, over $100$ realizations. The average value of the tuning parameters is reported in the last column. The benchmark ``estimators'' are labeled by a star. The significance level is $0.05.$
}
\begin{tabular}{rlccccc}
\multicolumn{7}{c}{Model 2}  \\[0.1cm]
\multicolumn{7}{c}{$p=100,n=400$}  \\[0.1cm]
  \hline
	&     & \multicolumn{2}{c}{Coverage}& \multicolumn{2}{c}{Length} & Average $\lambda$\\ 
 & Method & $S_0$&$S_0^c$  &$S_0$&$S_0^c$ \\ 
  \hline
  1 & glasso & 64.17 & 98.65 & 0.16 & 0.15& 0.067\\ 
  2 & glasso-weigh & 16.80 & 98.56 & 0.05 & 0.05 & 0.067\\ 
  3 & node-sqrt & 87.23 & 94.43 & 0.24 & 0.21 & 0.107\\ 
  4 & node-sqrt-tau & 89.81 & 97.23 & 0.20 & 0.18&0.107\\ 
  5 & node &  38.19 & 99.07 & 0.10 & 0.10&0.107\\ 
	6 & MLE & 50.98 & 91.22 & 0.30 & 0.26& -\\ 
	\hline
  7 & oracle* & 98.51 & - & 0.23 & 0.20 &-\\ 
  8 & perfect* & 95.00 & 95.00 & 0.22& 0.20&-\\ 
\end{tabular}
\label{table:m3}
\end{table}

\begin{table}
\centering
\caption[Average coverage and length of confidence intervals  over all  the entries and an average value of the tuning parameter $\lambda$.
]{
Average coverage and length of confidence intervals  over all  the entries and an average value of the tuning parameter $\lambda$.
The significance level is $0.05.$
}
\begin{tabular}{rlccc}
\multicolumn{5}{c}{Model 3: $\Theta_{ij} = 0.5^{i-j}$}  \\[0.1cm]
\multicolumn{5}{c}{$p=100,n=200$}  \\[0.1cm]
  \hline
	&   Method  & \multicolumn{1}{c}{Coverage}& \multicolumn{1}{c}{Length} & Average $\lambda$\\ 
  \hline
  1 & glasso & 90.43  & 0.19 & 0.138\\ 
  2 & glasso-weigh &75.81  & 0.33 & 0.138\\ 
  3 & node-sqrt & 93.36  & 0.28 & 0.152\\ 
  4 & node-sqrt-tau & 92.91  & 0.22 & 0.152\\ 
  5 & node & 89.88  & 0.20 & 0.152\\ 
	6 & MLE & 80.41  & 0.56 & -\\ 
	\hline
  7 & perfect* & 95.00 & 0.28 &- \\ 
\end{tabular}
\label{table:m4}
\end{table}



\subsection{Discussion}
\label{sec:con}
We have shown several constructions of asymptotically linear estimators of the precision matrix (and inverse correlation matrix) based on regularized estimators, 
which immediately lead to inference in Gaussian graphical models.
Efficient algorithms are available for both methods as discussed in Section \ref{sec:comput.view}.
The constructed estimators 
achieve entrywise estimation at the parametric rate and a rate of convergence of order $\sqrt{\log p/n}$ in supremum norm. 
\par
To provide a brief comparison of the two methods analyzed above, 
both theoretical and computational results seem in favor of the de-sparsified nodewise Lasso. 
Theoretical results for nodewise Lasso in the regime $p\gg n$ only need the mild conditions \conds \ref{eig}, \ref{subgv} and $d=o(\sqrt{n}/\log p)$ and are uniform over the considered model. Moreover, the de-sparsified nodewise Lasso may be thresholded again to yield 
recovery of the set $S$ with no false positives, and under a beta-min type condition, exact recovery of the set $S$, asymptotically, with high probability.
The graphical Lasso requires that we impose the strong irrepresentability condition in the high-dimensional regime.
However, the graphical Lasso might be preferred on the grounds that it does not decouple the likelihood. 
Moreover, the graphical Lasso estimator is always strictly positive definite and thus yields an estimator of the covariance matrix as well.
The invertibility of the nodewise Lasso has not yet been explored.
\par
We remark that the sparsity condition $d=o(\sqrt{n}/\log p)$ implied by our analysis is stronger than the condition needed for oracle inequalities and recovery, namely $d=o(n/\log p)$. However, one can show that this sparsity  condition is essentially necessary for asymptotically normal estimation. This follows by inspection of the minimax rates (see \cite{zhou}).



\section{Directed acyclic graphs}
\label{sec:dags}
In this section, we use the de-biasing ideas to construct confidence intervals for edge weights in directed acyclic graphs (abbreviated as DAGs).
A directed acyclic graph is a directed graph (we distinguish between edges $(j,k)$ and $(k,j)$) 
without directed cycles. 
We consider the Gaussian DAG model, where the DAG represents the probability distribution of a random vector  $(X_1,\dots,X_p)$ with a Gaussian distribution $ \mathcal N(0,\Sigma_0)$, where $\Sigma_0\in\mathbb R^{p\times p}$ is an unknown covariance matrix.
A Gaussian DAG may be 
represented by 
the linear structural equations model
$$X_j = \sum_{k\;\in \;{\text{pa}}(j)}\beta_{k,j}^0 X_k + \epsilon_j,\;\;\;\;\; j=1,\dots,p,$$
where $\epsilon_1,\dots,\epsilon_p$ are independent and $\epsilon_j\sim \mathcal N(0,(\omega^0_j)^2)$. 
The set ${\text{pa}}(j)$ is called the set of parents of a node $j$ and it contains all nodes $k\in\{1,\dots,p\}$ such that there exists a directed edge $k\rightarrow j$. 
\par
Our aim is to construct confidence intervals for edge weights $\beta_{k,j}^0$. However, without further conditions, 
the DAG and the $\beta_{k,j}^0$'s may not be identifiable from the structural equations model. 
To ensure identifiability, we assume that the error variances are equal: $\omega_j^0 = \omega_0$ for all $j=1,\dots,p.$ We remark that one might equivalently assume that the error variances are all known up to a multiplicative constant. 
In this setting, the DAG is identifiable as shown in \cite{jonas}. 
Our strategy is to use a two-step procedure: in the first step we use the estimator proposed in \cite{dag.sara} to estimate the ordering of the variables and in the second step, we use a de-biased version of nodewise regression to construct the confidence intervals.
\par
Given an $n\times p$ matrix $X=[X_1,\dots,X_p]$, with rows being $n$ independent observations from the structural equations model, one may rewrite the above model in a matrix form
$$X=XB_0+E,$$
where $B_0:=(\beta^0_{k,j})$ is a $p\times p $ matrix with $\beta^0_{j,j}=0$ for all $j$, and $E$ is an $n\times p$ matrix of noise vectors $E:=(\epsilon_1,\dots,\epsilon_p)$ with columns $\epsilon_j$ independent of $X_k$
whenever $\beta^0_{k,j}\not = 0.$ The rows of $E$ are independent $\mathcal N(0,\omega_0^2 I)$-distributed random vectors.
The model then implies that $X$ has covariance matrix
$$\Sigma_0=\omega_0^2((I-B_0)^{-1})^T (I-B_0)^{-1}.$$
We define the precision matrix (assumed to exist) by $\Theta_0:= \Sigma_0^{-1}.$ Notice that
$$\Theta_0 = \frac{1}{\omega_0^2} (I-B_0) (I-B_0)^T.$$
 We further consider the class of precision matrices corresponding to DAGs.
That is, we let
$$\Theta:= \Theta(B,\omega) = \frac{1}{\omega^2}(I-B)(I-B)^T,$$
where $(B,\omega)$ is such that there exists a DAG representing the distribution $\mathcal N(0,\Sigma)$ with
$\Sigma = \omega^2((I-B)^{-1})^T (I-B)^{-1}$. This means that $\omega>0$ and $B$ can be written as a lower-diagonal matrix, up to permutation of rows.
Further we let $s_B$ denote the number of nonzero entries in $B$, which corresponds to the number of edges in the DAG.
Moreover, we denote by $\mathcal B$ the set of all edge weights $B$ of DAGs with parameters $(B,\omega)$ which have at most $\alpha n/\log p$
incoming edges (parents) at each node, where $\alpha>0$ is given.

\subsection{Maximum likelihood estimator with $\ell_0$-penalization}
In the first step, we use an $\ell_0$-penalized maximum likelihood estimator to estimate the DAG.
Let $\hat\Sigma = X^T X/n$
be the Gram matrix based on the design matrix $X.$
The minus log-likelihood is proportional to $\ell(\Theta) = \text{trace}(\Theta \hat\Sigma) - \log \text{det}(\Theta).$
Consider the penalized maximum likelihood estimator proposed in \cite{dag.sara},
\begin{eqnarray}\nonumber
(\hat B,\hat \omega) &:=& \text{argmin}_{B,\omega}\{\;
\ell(B,\omega) + \lambda^2 s_B: \Theta = \Theta (B,\omega),\text{ for some DAG}
\\\label{l0}
&& \;\; \;\;\;\;\;\;\;\;\;\;\;\;\;\;\;\;\;\text{with parameters } (B,\omega) \text{ where } B\in \mathcal B
\},
\end{eqnarray}
where $\lambda\geq 0$ is a tuning parameter.
The estimator is denoted  by $\hat\Theta = \Theta(\hat B, \hat \omega)$ and it has $\hat s:= s_{\hat B} $ edges.
Calculating the $\ell_0$-penalized maximum likelihood estimator over the class of DAGs is a computationally intensive task,
especially because it involves a search through a class of DAGs under a non-convex constraint of acyclicity of the graph and due to the $\ell_0$-penalty. 
For large scale problems, greedy algorithms may be used, see e.g. \cite{chickering, hauser}.
The reason for using the $\ell_0$-penalty instead of $\ell_1$-penalization in the definition of \eqref{l0} was discussed in \cite{dag.sara}.
The $\ell_1$-penalty leads to an objective function which is not constant
over equivalent DAGs encoding the same distribution. The $\ell_0$-penalization leads to invariant scores over equivalent DAGs.
The theoretical properties of $\hat\Theta$ were studied in \cite{dag.sara} under the conditions summarized below. We remark that the paper \cite{dag.sara} primarily studies the estimator \eqref{l0} with unequal variances and shows that the estimator converges to some member of the Markov equivalence class (cf. \cite{pearl2000}) of a DAG with a minimal number of edges, under certain conditions.
\par
To make their result precise, we define some further notions. For any vector $\beta\in\mathbb R^p$, let $\|X\beta\|:=(\beta^T\Sigma_0\beta)^{1/2}.$ By an ordering of variables we mean any permutation of the set $\{1,\dots,p\}.$ 
For any ordering of the variables, $\pi$, we let $\tilde B(\pi)$ be the matrix obtained by doing a Gram-Schmidt orthogonalization 
of the columns of $X$ in the ordering given by $\pi$, with respect to the norm $\|\cdot\|.$
Moreover, let $\tilde\Omega_0 (\pi) = (I-\tilde B_0(\pi))^T\Sigma_0(I-\tilde B_0(\pi))=
\text{diag}((\tilde\omega_1^0(\pi))^2, \dots, (\tilde\omega_p^0(\pi))^2)$.
We restate the conditions assumed in \cite{dag.sara}.
\begin{condb}
\label{itm:eig}
There exists a universal constant $L\geq 1$ such that 
$$1/L \leq \Lambda_{\min}(\Sigma_0) \leq \Lambda_{\max}(\Sigma_0) \leq L.$$
\end{condb}
\begin{condb}
\label{itm:omega}
 There exists a constant $\eta_\omega>0$ such that for all $\pi$ such that $\tilde\Omega_0(\pi)\not = \omega_0^2 I$ it holds
$$\frac{1}{p}\sum_{i=1}^p (|\tilde\omega_j^0(\pi)|^2 - \omega_0^2)^2 > 1/\eta_\omega. $$
\end{condb}
\begin{condb}
\label{itm:ld}
There exists a sufficiently small constant $\alpha_*$ such that
$p\leq \alpha_* n/\log p.$
\end{condb}
Condition \ref{itm:omega} is an ``omega-min'' condition:
it imposes that if one uses the wrong permutation then the error variances are far enough from being equal. 
\par
Under the above conditions, the $\ell_0$-penalized maximum likelihood estimator with high-probabi\-li\-ty correctly estimates the 
ordering of the variables as shown in \cite{dag.sara}.
Let $\pi_0$ be an ordering  of the variables such that a Gram-Schmidt orthogonalization of the columns of $X$ in the order given by $\pi_0$ with respect to the norm $\|\cdot\|$ yields $B_0.$
Denote the ordering of variables estimated by the $\ell_0$-penalized maximum likelihood estimator by $\hat \pi.$
Then the result in \cite{dag.sara} states that under \conds \ref{itm:eig}, \ref{itm:omega} and \ref{itm:ld}, with high-probability it holds that $\hat \pi = \pi_0.$

\subsection{Inference for edge weights}
Given that we have recovered the true ordering $\pi_0$, the problem reduces to estimation of regression coefficients in 
a nodewise regression model, where each variable is a function of a known set of its ``predecessors''.
Therefore to construct asymptotically normal estimators for the $\beta_{k,j}^0$'s, we may use a nodewise regression approach.

An estimated ordering $\hat \pi$  yields estimates $\predh$ of the predecessor sets for each node $j=1,\dots,p.$
If we have recovered the true ordering, that is $\hat\pi = \pi_0$, the estimated predecessor sets  $\predh$ are equal to the true predecessor sets, which are supersets of the parent sets $\text{pa}(j)$ for each $j=1,\dots,p$.
Consequently, 
given the predecessor sets, we may obtain a new estimator for the edge weights
by regressing the $j$-th variable $X_j$ on all its predecessors. 
The predecessor sets $\predh$ might be as large as $p-1$, therefore it is necessary to use regularization.
Then we use the de-biasing technique in a similar spirit as in Section \ref{sec:debias}.
We remark that in the initial step, one may use any estimator which guarantees exact recovery of the ordering $\pi_0$.

\par
For any non-empty subset $T\subseteq \{1,\dots,p\},$ we denote by $X_{T}$ the $n\times |T|$ matrix formed by taking the columns $X_k$ of $X$ such that $k\in T.$
We define the nodewise regression estimator as proposed in \cite{jvdgeer15} (altenatively, one may use the nodewise square-root Lasso studied in Section \ref{sec:node})
\vskip 0.0cm
\begin{equation}\label{beta.dag}
\hat\beta_{{j}} = \operatornamewithlimits{argmin}\limits_{\beta\in\mathbb R^{|\predh|}}\|X_j - X_{\predh}\beta\|_2^2 / n + 2\lambda_j \|\beta\|_1.
\end{equation}
\vskip 0.2cm
\noindent
The Karush-Kuhn-Tucker conditions for the above optimization problem give
$$-X_{\predh}^T(X_j - X_{\predh}\hat\beta_{{j}})/n+\lambda_j\hat Z_j=0,$$
where the entries of $\hat Z_j$ satisfy $\hat Z_{k,j} = \text{sign}(\hat\beta_{k,j})$ if $\hat\beta_{k,j} \not =0$,  and $\|\hat Z_j\|_\infty \leq 1$ ($\hat\beta_{k,j}$ denotes the $k$-th entry of $\hat\beta_j$).
Similarly as in the case of undirected graphical models, we can define a de-biased estimator. 
The Hessian matrix of the risk function in \eqref{beta.dag} is given by
\vskip 0.1cm
$$\hat\Sigma_{\predh} := X_{\predh}^TX_{\predh}/n.$$
\vskip 0.1cm
To find a surrogate inverse for $\hat\Sigma_{\predh}$, we construct
 $\hat\Theta_{\predh}$ 
 using the nodewise Lasso with tuning parameters $\lambda_{k,j}$ for $k\in{\predh}$. 
Using $\hat\Theta_{\predh}$, we define the de-biased estimator
\vskip 0.1cm
\begin{equation}\label{b}
\hat b_{{j}}: =
\hat\beta_{{j}} + \hat\Theta_{\predh}^T X_{\predh}^T(X_j - X_{\predh}
\hat\beta_{{j}})/n.
\end{equation}
\vskip 0.1cm
Theorem \ref{dag} below shows that the entries of the de-biased estimator are asymptotically normal.
To formulate the result, we define $\Theta^0_{\predz}$ to be the matrix obtained by taking the rows and columns of $\Theta_0$ contained in 
the true predecessor set $\predz$. Denote the $k$-th column of $\Theta^0_{\predz}$ by $\Theta^0_{\predz,k}$.
%
To provide asymptotically normal estimators for the parameters $\beta_{\predz}^0 =(\beta^0_{k,j}:k\in\predz)$, we need to impose a sparsity condition on the sizes of the parent sets, which will be denoted by $d_j =|\text{pa}(j)|$. 
\begin{theorem}[Regime $p\leq n$]\label{dag}
Let $\hat B$ be the  estimator defined by \eqref{l0} with $\lambda\asymp \sqrt{\log p /n}$ and denote the  predecessor sets estimated based on $\hat B$ by $ \hat{{\text{p}}}(j)$ for $ j=1,\dots,p.$
Let $\hat b_{{j}}$ be defined in \eqref{b}  with sufficiently large tuning parameters $\lambda_j\asymp \lambda_{k,j}\asymp\sqrt{\log p /n}$, uniformly in $j,k$, where $k\in\predh$.
Assume \conds \ref{itm:eig}, \ref{itm:omega} and \ref{itm:ld} are satisfied with  $1/(|\alpha_*| + |\eta_{\omega}|)=\mathcal O(1)$ and assume that $d_j =o(\sqrt{n}/\log p)$.
Then it holds
\begin{equation}
 \hat b_{{j}} -  \beta_{\predz}^0 \;=\; (\Theta_{\predz}^0)^T X_{\predz}^T \epsilon_j/n\; + \;\emph{rem},
\end{equation}
where
$$\|\emph{rem} \|_\infty 
 =o_P(1/\sqrt{n}).$$
Furthermore, for every $k\in \predz,$
$$\sqrt{n}(\hat b_{k,j} - \beta^0_{k,j})/\sigma_{k,j} \rightsquigarrow \mathcal N(0,1 ),$$
where the asymptotic variance of the de-sparsified estimator is given by 
$$\sigma_{k,j}^2:=
 n\emph{var}((\Theta^0_{\predz,k})^TX_{\predz}^T \epsilon_{j})
= \omega_0^2 (  \Theta^0_{\predz}   )_{kk}.$$
\end{theorem}

\noindent
The result of Theorem \ref{dag} can be used to construct confidence intervals for the edge weights $\beta_{k,j}^0$. 
To estimate the asymptotic variance, we may define 
$\hat\omega_j^2:=\|X_j - X_{\widehat{\text{p}}(j)}\hat\beta_{j}\|_2^2/n$ and $\hat\sigma^2_{k,j}: = \hat\omega_j^2 
(\hat \Theta_{\predh})_{kk} .$ The consistency of this estimator may be easily checked. 


\section{Conclusion}
We have provided a unified approach to construct asymptotically linear and normal estimators of low-dimensional parameters of the precision matrix based on regularized estimators. These estimators allow us to construct confidence intervals for edge weights in high-dimensional Gaussian graphical models and, under an identifiability condition, for edge weights in the high-dimensional Gaussian DAG model.
\par 
For Gaussian graphical models, we provided two explicit simple constructions: one based on a global method using the graphical Lasso
and the second based on a local method using nodewise Lasso regressions. Efficient computational methods are available for both methods as discussed in Section \ref{sec:comput.view}.
The constructed estimators are asymptotically normal per entry, achieving the efficient asymptotic variance from the parametric setting. For a detailed analysis of semi-parametric efficiency bounds in Gaussian graphical models, we refer to \cite{ae}. 
For testing hypothesis about a set of edges, the usual multiple testing corrections may be used although in practical applications, these might turn out to be too conservative. More efficient methods for multiple testing in this setting are yet to be developed. 
While throughout the presented results we have imposed ``exact'' sparsity constraints on the underlying parameters, we remark that the results might as well be extended to models which are only approximately sparse (see e.g. \cite{hds}).
\par
Our main interest lied in developing methodology for graphical models representing continuous random vectors. However, many applications involve \emph{discrete}  graphical models, where random variables $X_j$ at each vertex $j\in  \mathcal V$ take values in a discrete space.
A popular family of distributions for the binary case where $X_j\in \{-1,1\}$  is the Ising model. 
This model finds applications in statistical physics, neuroscience or modeling of social networks.
The Ising model can be efficiently estimated via a nodewise method: the individual neighbourhoods can be estimated with $\ell_1$-penalized logistic regression as proposed in \cite{ravikumar2010}. 
Logistic regression falls into the framework of generalized linear models for which the de-biasing methodology was proposed in \cite{vdgeer13}. Consequently, one may compute the neighbourhood estimator via $\ell_1$-penalized logistic regression and then compute the de-biased estimator along the lines of \cite{vdgeer13}.

\par
For directed acyclic graphs, we showed that confidence intervals for edge weights may be constructed for the Gaussian DAG when it is identifiable and $p\leq \alpha_* n/\log p$. To this end, we require that the error variances in the structural equations model are equal, or known up to a multiplicative constant. 
If the variance are not equal, the model may not be identifiable and work on inference in this setting is yet to be developed.

\section{Proofs}
\subsection{Proofs for undirected graphical models}

\begin{lemma}\label{taylor}
Assume that $1/L\leq \Lambda_{\min}(\Theta_0) \leq \Lambda_{\max}(\Theta_0)\leq L$ for some constant $L\geq 1.$
Let $\mathcal E (\Delta) := \emph{tr}[\Delta\Sigma_0] 
- [\log \emph{det}(\Delta+\Theta_0) - \log \emph{det}(\Theta_0)]$.
Then for all $\Delta$ such that $\|\Delta\|_F \leq 1/(2L)$, $\mathcal E(\Delta)$ is well defined and 
\begin{equation}
\label{deto}
\mathcal E(\Delta) \geq \frac{1}{2({L} + 1/(2L))^2}\|\Delta\|_F^2.
\end{equation}

\end{lemma}
\begin{proof}[Proof of Lemma \ref{taylor}]
First we show that $\mathcal E(\Delta)$ is well defined for all $\Delta$ such that $\|\Delta\|_F \leq 1/(2L).$
To this end, we need to check that $\Lambda_{\min}(\Theta_0 + \Delta)\geq c_1$ for some $c_1>0.$ Denote the spectral norm of a matrix $M$ by $\|M\|:=\sqrt{\Lambda_{\max}(MM^T)}.$
We have
$$\Lambda_{\min}(\Theta_0 + \Delta) = \min_{\|x\|_2=1}x^T(\Theta_0 + \Delta) x
\geq 
\Lambda_{\min}(\Theta_0) - \|\Delta\|_F
\geq 1/(2L),
$$
where we used that $|x^T\Delta x|\leq \|\Delta\|x^Tx$ and that $\|\Delta\|\leq \|\Delta\|_F$.
\\
A second order Taylor expansion with remainder in integral form yields
\begin{eqnarray*}
&&\log \text{det}(\Delta+\Theta_0) - \log \text{det}(\Theta_0) 
\\
&&=
\text{tr}(\Delta\Sigma_0) - \text{vec}(\Delta) ^T 
\left(\int_0^1 (1-v)(\Theta_0 + v \Delta)^{-1} \otimes (\Theta_0 + v \Delta)^{-1} dv\right)
\text{vec}(\Delta).
\end{eqnarray*}
Then for all $\Delta$ such that $\|\Delta\|_F\leq 1/(2L),$ it holds
$$\mathcal E(\Delta) = \text{vec}(\Delta) ^T 
\left(\int_0^1 (1-v)(\Theta_0 + v \Delta)^{-1} \otimes (\Theta_0 + v \Delta)^{-1} dv\right)
\text{vec}(\Delta),$$
where $\otimes $ denotes the Kronecker product and the remainder in the Taylor expansion is in the integral form.
Using the fact that the eigenvalues of Kronecker product of
symmetric matrices is the product of eigenvalues of the factors, it follows for all 
$\Delta$ such that $\|\Delta\|_F\leq 1/(2L)$ 
\begin{eqnarray*}
&&\Lambda_{\min}\left(\int_0^1 (1-v)(\Theta_0 + v \Delta)^{-1} \otimes (\Theta_0 + v \Delta)^{-1} dv\right)
\\
&\geq & 
\int_0^1  (1-v) \Lambda_{\min}^2((\Theta_0 + v \Delta)^{-1}) dv\\
&\geq  &
\frac{1}{2} \min_{0\leq v \leq 1}\Lambda_{\min}^2((\Theta_0 + v \Delta)^{-1})\\
&\geq &
\frac{1}{2} \min_{\Delta:\|\Delta\|_F\leq 1/(2L)}\Lambda_{\min}^2((\Theta_0 + \Delta)^{-1})
.
\end{eqnarray*}
Next we obtain
%
%
\begin{eqnarray*}
 \Lambda_{\min}^2((\Theta_0 + \Delta)^{-1}) = \Lambda_{\max}^{-2}(\Theta_0 + \Delta)
\geq 
(\|\Theta_0\| + \|\Delta\|)^{-2} \geq \frac{1}{({L} + 1/(2L))^2}>0,
\end{eqnarray*}
where we used $\|\Delta\|\leq \|\Delta\|_F\leq 1/(2L)$. 
Finally this yields that
$\mathcal E(\Delta)\geq \frac{1}{2({L} + 1/(2L))^2} \|\Delta\|_F^2$ for all $\|\Delta\|_F\leq 1/(2L)$, as required.
\end{proof}

\begin{proof}[Proof of Theorems \ref{glasso.rate} and \ref{norm.glasso.rate}]
We will prove both Theorem \ref{glasso.rate} and Theorem \ref{norm.glasso.rate} at the same time, since the proofs only differ slightly.
For the proof of Theorem \ref{norm.glasso.rate}, one has to replace $\hat\Sigma$, $\Sigma_0$, $\hat\Theta,$ $\Theta_0$ in the proof below by $\hat\Gamma$, $\Gamma_0$, $\hat\Theta_{\text{norm}},$ $K_0$, respectively. 
\vskip 0.2cm
Let $\tilde\Theta:= \alpha \hat\Theta+(1-\alpha)\Theta_0$, where $\alpha := \frac{M}{M+\|\hat\Theta-\Theta_0\|_F},$ for some $M>0$ to be specified later. 
The definition of $\tilde\Theta$ implies that $\|\tilde\Theta-\Theta_0\|_F \leq M.$
By the convexity of the loss function and by the definition of $\hat\Theta$, we have
\begin{equation}
\label{ineq0}
\text{tr}(\tilde\Theta\hat\Sigma) -\log \text{det}(\tilde\Theta)  +\lambda\|\tilde\Theta^-\|_1 
 \leq \text{tr}(\Theta_0\hat\Sigma) -\log \text{det}(\Theta_0) +\lambda\|\Theta^-_0\|_1.
\end{equation}
Denote $\Delta=\tilde\Theta-\Theta_0$ and let
$$\mathcal E(\Delta) := \text{tr}(\Delta\Sigma_0) 
- [\log \text{det}(\Delta+\Theta_0) - \log \text{det}(\Theta_0)].$$
The inequality \eqref{ineq0} implies the basic inequality
\begin{eqnarray*}
\mathcal E(\Delta)
+\lambda\|\tilde\Theta^-\|_1 
\leq 
-\text{tr}[\Delta(\hat\Sigma-\Sigma_0)] 
+ \lambda \|\Theta_0^-\|_1
.
\end{eqnarray*}
On the set $\{\|\hat\Sigma-\Sigma_0\|_\infty \leq \lambda_0\}$, we can bound the empirical process term by 
\begin{eqnarray}
\nonumber
|\text{tr}[\Delta(\hat\Sigma-\Sigma_0)] |
 &\leq & 
\|\hat\Sigma^--\Sigma_0^-\|_\infty \|\Delta^-\|_1 + \|\hat\Sigma^+ - \Sigma_0^+\|_2\|\Delta^+\|_F
\\\label{rando}
&\leq &
\lambda_0 \|\Delta^-\|_1 + \|\hat\Sigma^+ - \Sigma_0^+\|_2\|\Delta^+\|_F
. 
\end{eqnarray}
In what follows, we work on the set $\{\|\hat\Sigma-\Sigma_0\|_\infty \leq \lambda_0\}$.
\\
We now choose $M$ such that $M\leq 1/(2L),$ this then implies 
$\|\tilde\Theta-\Theta_0\|_F\leq 1/(2L).$ But then Lemma \ref{taylor} implies that 
$\mathcal E(\tilde\Theta-\Theta_0)$ is well defined and  
\begin{equation}
\label{det}
\mathcal E(\tilde\Theta-\Theta_0) \geq c\|\tilde\Theta-\Theta_0\|_F^2,
\end{equation}
where one can take $c:=1/(8L^2).$
Using  bounds \eqref{rando} and \eqref{det}, we obtain from the basic inequality
\begin{eqnarray*}
&& c \|\Delta\|_F^2 
+\lambda \|\tilde\Theta^-\|_1 \leq \lambda_0\|\Delta^-\|_1 +\lambda \|\Theta_0^-\|_1 + \|\hat\Sigma^+ - \Sigma_0^+\|_2\|\Delta^+\|_F
\end{eqnarray*}
By the triangle inequality and taking $\lambda \geq 2\lambda_0$, we obtain
\begin{eqnarray*}
&& 2c \|\Delta\|_F^2 
+\lambda \|\tilde\Theta_{S^c}^-\|_1 \leq 3\lambda \|\Delta_S^-\|_1  + 2\|\hat\Sigma^+ - \Sigma_0^+\|_2\|\Delta^+\|_F
\end{eqnarray*}
Consequently,
\begin{eqnarray*}
2c \|\Delta\|_F^2 
+\lambda \|\Delta^-\|_1  
&\leq & 4\lambda \|\Delta_S^-\|_1 + 2\|\hat\Sigma^+ - \Sigma_0^+\|_2\|\Delta^+\|_F
\\
&\leq & 4\lambda \sqrt{s} \|\Delta^-_S\|_F+ 2\|\hat\Sigma^+ - \Sigma_0^+\|_2\|\Delta^+\|_F
\\
&\leq &  
8 s\lambda ^2/c^2 + c\|\Delta^-_S\|_F^2/2 + 8\|\hat\Sigma^+ - \Sigma_0^+\|_2^2/c + c\|\Delta^+\|_F^2/2.
\end{eqnarray*}
Taking $M$ such that $\lambda_0 M \geq 8 s\lambda ^2/c^2 +8\|\hat\Sigma^+ - \Sigma_0^+\|_2^2/c,$
\begin{eqnarray*}
&& c \|\Delta\|_F^2 
+\lambda \|\Delta^-\|_1  \leq 8 s\lambda ^2/c^2 +8\|\hat\Sigma^+ - \Sigma_0^+\|_2^2/c\leq \lambda_0 M.
\end{eqnarray*}
Taking $M\geq 4\lambda_0/c$,
\begin{eqnarray*}
&& \|\Delta\|_F^2 \leq \lambda_0 M/c\leq M^2/4.
\end{eqnarray*}
But then $\|\Delta\|_F\leq M/2.$
The definition of $\tilde\Theta$ in turn implies that $\|\hat\Theta-\Theta_0\|_F \leq M,$
and we can repeat all the arguments with $\hat\Theta$ in place of $\tilde\Theta.$
Repetition of the arguments leads to the oracle inequality
\begin{eqnarray*}
&& c \|\hat\Theta-\Theta_0\|_F^2 
+\lambda \|\hat\Theta^--\Theta_0^-\|_1  \leq 8 s\lambda ^2/c^2 +8\|\hat\Sigma^+ - \Sigma_0^+\|_2^2/c.
\end{eqnarray*}
Finally we distinguish the case of non-normalized graphical Lasso (based on the covariance matrix)
and the normalized graphical Lasso (based on the correlation matrix). We have for the case of
\begin{enumerate}[a)]
\item
\underline{normalized graphical Lasso}: $\hat\Sigma^+ - \Sigma_0^+=0$ (recall here that $\hat\Sigma\equiv\hat R, \Sigma_0\equiv R_0$) and the oracle inequality gives
\begin{eqnarray*}
&& c \|\hat\Theta-\Theta_0\|_F^2 
+\lambda \|\hat\Theta^--\Theta_0^-\|_1  \leq 8 s\lambda ^2/c^2.
\end{eqnarray*}

\item
\underline{non-normalized graphical Lasso}:
we can bound
 $$\|\hat\Sigma^+ - \Sigma_0^+\|_2 \leq \sqrt{p}\|\hat\Sigma^+ - \Sigma_0^+\|_\infty \leq \sqrt{p}\lambda_0.$$
Hence the oracle inequality gives 
\begin{eqnarray*}
&& c \|\hat\Theta-\Theta_0\|_F^2 
+\lambda \|\hat\Theta^--\Theta_0^-\|_1  \leq 8 s\lambda ^2/c^2 +8{p}\lambda_0^2/c.
\end{eqnarray*}
\end{enumerate}
To show the second statement of the theorems, we use the above oracle inequalities and the following upper bound
\begin{eqnarray*}
 \vertiii{\hat\Theta_{\text{}}-\Theta_0}_1 
&\leq &
\max_{j=1,\dots,p}|\hat\Theta_{jj}-\Theta_{jj}^0| + \|\hat\Theta_j^- - (\Theta_j^0)^-\|_1
\\
&\leq &
\|\hat\Theta-\Theta_0\|_F + \|\hat\Theta^- - \Theta_0^-\|_1
.
\end{eqnarray*}
To show the third statement of Theorem \ref{norm.glasso.rate}, we use the upper bound
\begin{eqnarray*}
\vertiii{\hat\Theta_{\text{w}} -\Theta_0}_1 &=&
\vertiii{\hat W^{-1} \hat\Theta_{\text{norm}}\hat W ^{-1}- W_0^{-1}K_0 W_0^{-1}}_1
\\
&\leq & \|\hat W \|^2_\infty \vertiii{\hat\Theta_{\text{norm}} - K_0}_1 + \|\hat W -W_0\|_\infty \vertiii{K_0}_1 \|\hat W\|_\infty\\
&& + \|W_0\|_\infty \vertiii{K_0}_1\|\hat W -W_0\|_\infty
\end{eqnarray*}

\end{proof}

\begin{proof}[Proof of Theorem \ref{glasso.rem}]
Denote $C_L:=16(8L^2)^2.$
Using the results of Theorem \ref{glasso.rate}, we obtain
\begin{eqnarray*}
\|\text{rem}\|_\infty
&\leq & 
\|(\hat\Theta_{}-\Theta_0)^T(\hat \Sigma \Theta_0- I) \|_\infty
+
\|(\hat\Theta_{}-\Theta_0)^T\lambda \hat Z\hat \Theta\|_\infty
\\
&\leq &
\vertiii{\hat\Theta_{}-\Theta_0}_1 \|\hat \Sigma - \Sigma_0\|_\infty \vertiii{\Theta_0}_1
+
 \vertiii{\hat\Theta_{}-\Theta_0}_1 \lambda \|\hat Z\|_\infty \vertiii{\hat \Theta}_1
\\
&\leq &
C_L (p+s)\lambda \sqrt{d+1}\Lambda_{\max}(\Theta_0)\lambda_0 + 2C_L (p+s)\lambda  \sqrt{d+1}\Lambda_{\max}(\Theta_0) \lambda\\
&\leq &
\frac{3}{2}C_L L (p+s)\sqrt{d+1}\lambda^2.
\end{eqnarray*}
Taking $\lambda\asymp\sqrt{\log p/n},$ Lemma \ref{conc1} implies $\|\hat \Sigma - \Sigma_0\|_\infty = \mathcal O_P(\sqrt{\log p/n}).$
Then by the sparsity condition, we obtain $\|\text{rem}\|_\infty = \mathcal O_P(1/\sqrt{n}).$
By \conds \ref{eig} and \ref{subgv}, the random variable
$(\Theta_0(\hat\Sigma-\Sigma_0)\Theta_0)_{ij}$ has bounded fourth moments and asymptotic normality per entry follows by application of Lindeberg's central limit theorem for triangular arrays (see \cite{jvdgeer15} for more details).
\end{proof}

\begin{proof}[Proof of Theorem \ref{denorm.glasso.rem}]
For the remainder, we obtain similarly as in the proof of Theorem \ref{norm.glasso.rem},
$\|\text{rem}\|_\infty
\leq \frac{3}{2}C_L L s^{}\sqrt{d+1}\lambda^2.$
Asymptotic normality follows by analogous arguments.
\end{proof}

\begin{proof}[Proof of Proposition \ref{norm.glasso.rem}]
Denote $C_L:=16(8L^2)^2.$
Using the results of Theorem \ref{norm.glasso.rate}, we obtain
\begin{eqnarray*}
\|\text{rem}\|_\infty
&\leq &
\vertiii{\hat\Theta_{\text{norm}}-K_0}_1 \|\hat\Gamma K_0- I\|_\infty
+
 \vertiii{\hat\Theta_{\text{norm}}-K_0}_1 \lambda \|\hat Z\|_\infty\vertiii{\hat \Theta_{\text{norm}}}_1
\\
&\leq &
C_L s^{}\lambda\sqrt{d+1}\Lambda_{\max}(\Theta_0)  \lambda_0 + C_L s\lambda\sqrt{d+1}\Lambda_{\max}(\Theta_0)\lambda
\\
&\leq & \frac{3}{2}C_L L s^{}\sqrt{d+1}\lambda^2.
\end{eqnarray*}
The sparsity condition implies the result.
\end{proof}

\begin{proof}[Proof of Theorem \ref{node.rem}]
The proof follows along the same lines as the proof of Theorem 1 in \cite{jvdgeer15}. The only difference is that here we consider a weighted Lasso to estimate the partial correlations and the estimator $\hat\tau_j$ is defined slightly differently. But for the weighted Lasso (with weights bounded away from zero and bounded from above with high probability), oracle inequalities of the same order can be obtained, 
see Section 6.9 in \cite{hds}, i.e.
$$\|X_{-j}(\hat\gamma_j - \gamma_j^0)\|_2^2/n  + \lambda\|\hat\gamma_j - \gamma_j^0\|_1=\mathcal O_P(d_j{\log p/n}).$$
For the estimator of variance we have
\begin{eqnarray*}
|\hat\tau_j^2 -\tau_j^2| &\leq &
 \|X_{-j}(\hat\gamma_j - \gamma_j^0)\|_2^2/n + 2 |(X_j-X_{-j}\gamma_j^0)^TX_{-j}(\hat\gamma_j-\gamma_j^0)/n|
\\
&=& 
\|X_{-j}(\hat\gamma_j - \gamma_j^0)\|_2^2/n + 2 \|(X_j-X_{-j}\gamma_j^0)^TX_{-j}/n\|_\infty \|\hat\gamma_j-\gamma_j^0\|_1
\\
&=&
\mathcal O_P(1/\sqrt{n}).
\end{eqnarray*}
The rest of the proof follows as in \cite{jvdgeer15}.

\end{proof}

\subsection{Proofs for directed acyclic graphs}

\begin{proof}[Proof of Theorem \ref{dag}]
By Theorem 5.1 in \cite{dag.sara}, we have under the conditions of the theorem that $\hat\pi = \pi_0$ with high probability. 
Then also 
$\predh = \predz$ for all $j$, with high probability. Therefore, the estimated $\predh$ in the definitions of $\hat\beta_j$ and $\hat\Theta_{\predh,k},k\in \predh$ (and elsewhere) can be replaced by $\predz.$
The nodewise Lasso then yields oracle estimators $\hat\beta_j$ and $\hat\Theta_k,k\in\predz$ under the \condi \ref{itm:eig} 
and under the sparsity $d_j=o(\sqrt{n/\log p})$ (see \cite{jvdgeer15}).
This gives in particular that  for all $j=1,\dots,p$
$$\max_{k}\|\hat\Theta_{\predz,k}-\Theta^0_{\predz,k}\|_1  =\mathcal O_P(\max_{k}d_j\lambda_j),$$
$$
\max_{k}\|\hat\Sigma_{\predz}\hat\Theta_{\predz,k} - e_k\|_\infty =\mathcal O_P(\max_{k}\lambda_j),$$
$$\|\hat\beta_j - \beta_j^0\|_1 = \mathcal O_P(\max_{j=1,\dots,p}d_j\lambda_j).
 $$
We can write the decomposition 
\begin{eqnarray}
\label{deco.dag}
 \hat b_{k,j} -  \beta_{k,j}^0 = (\Theta_{\predz,k}^0)^T X_{\predz}^T \epsilon_j/n + \text{rem}_{k,j},
\end{eqnarray}
where $\text{rem}_{k,j}=(\hat\Theta_{\predz,k}-\Theta^0_{\predz,k})^T X_{\predz}^T \epsilon_j/n - 
(\hat\Sigma_{\predz}\hat\Theta_{\predz,k} - e_k)^T(\hat\beta_{j}-\beta_j^0).$ 
First note that by normality and by the independence of $X_{\predz}$ and $\epsilon_j$ (which follows by the independence of $\epsilon_j$'s 
and acyclicity of the graph), it holds
$\|X_{\predz}^T \epsilon_j/n\|_\infty = \mathcal O_P(\sqrt{\log p/n}).$
By H\"older's inequality
\begin{eqnarray*}
\max_{k}|\text{rem}_{k,j}|
&\leq &
\max_k\|\hat\Theta_{\predz,k}-\Theta^0_{\predz,k}\|_1 \|X_{\predz}^T \epsilon_j/n\|_\infty \\
&& + \;\;\max_k\|\hat\Theta_{\predz,k}^T\hat\Sigma_{\predz} - e_k\|_\infty \|\hat\beta_j - \beta_j^0\|_1
\\
&= &
\mathcal O_P(d_j\log p/n) = o_P(1/\sqrt{n}),
\end{eqnarray*}
where we used the sparsity assumption $d_j=o(\sqrt{n}/\log p)$.
Thus we have shown that the remainder in \eqref{deco.dag} is of small order $1/\sqrt{n}$.
Then applying Lindeberg's central limit theorem for triangular arrays and by \conds \ref{eig} and \ref{subgv},
$$(\Theta_{\predz,k}^0)^T X_{\predz}^T \epsilon_j/(\sigma_{k,j} \sqrt{n}) \rightsquigarrow \mathcal N(0,1),$$
which shows the claim.
\end{proof}